\documentclass[12pt]{article}

\pdfoutput=1

\usepackage{amsmath,amsthm,amssymb}
\usepackage{graphicx}

\def\S{\mathbb S}
\def\R{\mathbb R}
\def\L{\textup{Lk}}

\def\c{\mathsf \Lambda}
\def\s{\sigma}

\newtheorem{theorem}{Theorem}[section]
\newtheorem{lemma}[theorem]{Lemma}
\newtheorem{corollary}[theorem]{Corollary}
\newtheorem{proposition}[theorem]{Proposition}

\newtheorem{introthm}{Theorem}
\newtheorem{introcor}[introthm]{Corollary}

\theoremstyle{definition}
\newtheorem{definition}[theorem]{Definition}
\newtheorem{remark}[theorem]{Remark}
\newtheorem{example}[theorem]{Example}
\newtheorem{convention}[theorem]{Convention}

\numberwithin{equation}{subsection}
\numberwithin{figure}{section}

\begin{document}

\title{A convexity theorem for real projective structures}

\author{Jaejeong Lee}

\date{}

\maketitle

\begin{abstract}
Given a finite collection $\mathcal{P}$ of convex $n$-polytopes in
$\R\textup{P}^n$ ($n\ge2$), we consider a real projective manifold
$M$ which is obtained by gluing together the polytopes in
$\mathcal{P}$ along their facets in such a way that the union of any
two adjacent polytopes sharing a common facet is convex. We prove
that the real projective structure on $M$ is
\begin{enumerate}
\item
convex if $\mathcal{P}$ contains no triangular polytope, and
\item
properly convex if, in addition, $\mathcal{P}$ contains a polytope
whose dual polytope is thick.
\end{enumerate}
Triangular polytopes and polytopes with thick duals are defined as
analogues of triangles and polygons with at least five edges,
respectively.
\end{abstract}

\tableofcontents

\section{Introduction}

\begin{figure}[htbp]
\includegraphics[width=1\linewidth]{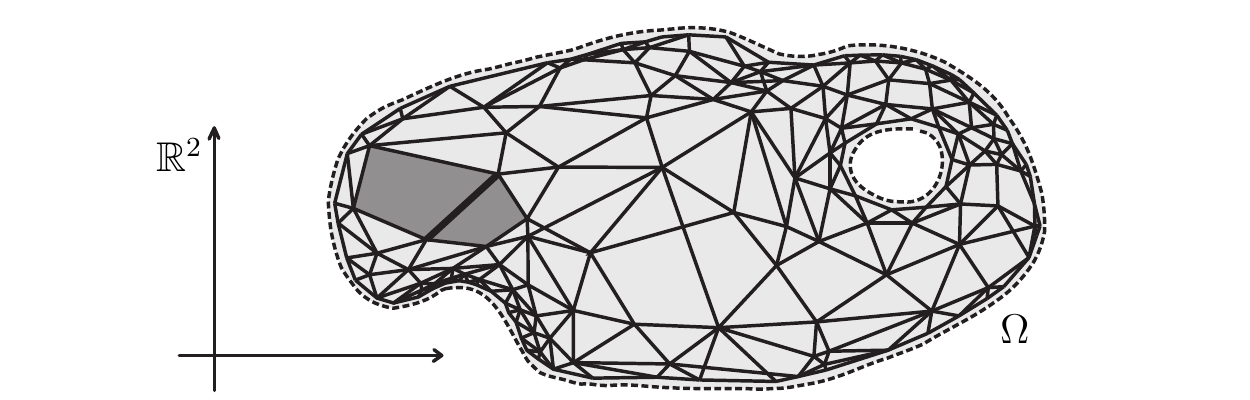}
\caption{\footnotesize A feasible picture of a planar domain
$\Omega$ with a residually convex tessellation. Shaded is the union
of two polygons sharing the thick common edge.} \label{fig:domain}
\end{figure}

Consider a planar domain $\Omega$, an open connected subset of
$\R^2$. Suppose that $\Omega$ admits a tessellation $\mathcal{T}$ by
(a necessarily infinite number of) convex polygons. One may ask if
there are any local conditions on the tessellation $\mathcal{T}$
which can guarantee convexity of the domain $\Omega$. One reasonable
such condition we investigate in this paper is the following:
\begin{quote}
\textsl{the union of two adjacent polygons sharing a common edge is
convex.}
\end{quote}
See Figure~\ref{fig:domain}. This condition was first introduced by
Kapovich \cite{kapovich} and we call tessellations with this
property \emph{residually convex}. It turns out that, under the
residual convexity condition, one can prove the following:
\renewcommand{\labelenumi}{(\Roman{enumi})}
\begin{enumerate}
\item
If $\mathcal{T}$ contains no triangle then the domain $\Omega$ is a
convex subset of $\R^2$.
\item
If, in addition, $\mathcal{T}$ contains a polygon with at least $5$
edges then the convex domain $\Omega$ contains no infinite line.
\end{enumerate}

Figure~\ref{fig:pentagon} illustrates the above assertions: (a)
exhibits a generic shape of a convex domain which admits a
residually convex tessellation without triangles, (b) shows that a
domain containing an infinite line may admit a residually convex
tessellation without polygons with at least $5$ edges, and (c) shows
that a domain with residually convex tessellation containing a
pentagon but no triangles is bounded.
\begin{figure}[htbp]
\includegraphics[width=1\linewidth]{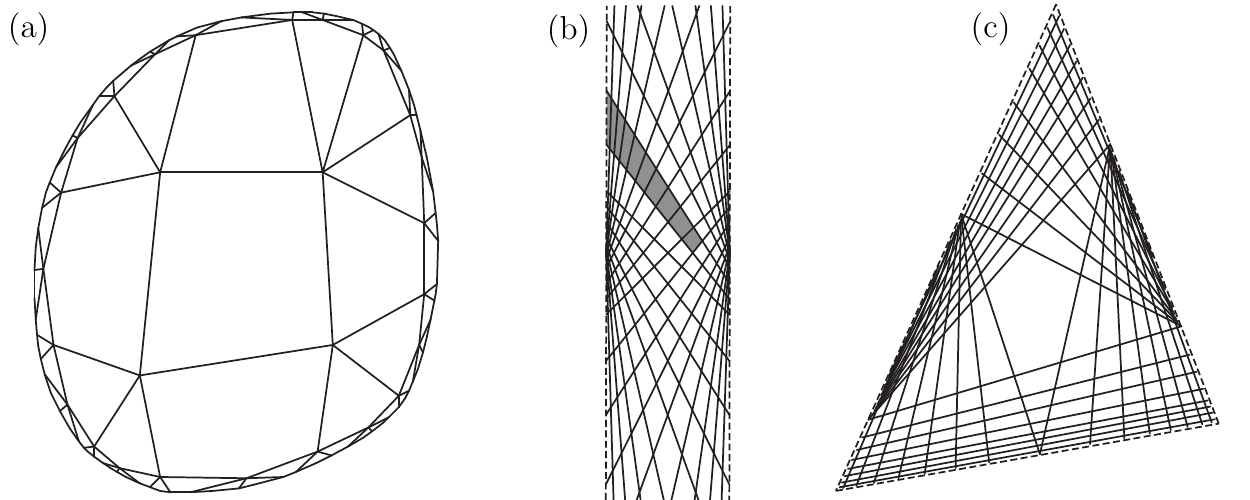}
\caption{\footnotesize (a) A generic residually convex tessellation
without triangles. (b) An unbounded domain with residually convex
tessellation by quadrilaterals, which is not the standard
tessellation of $\R^2$ by squares. A directed gallery (see
Definition~\ref{def:gal}) is shaded. (c) A maximal domain with
residually convex tessellation containing the pentagon in the middle
but no triangles.} \label{fig:pentagon}
\end{figure}

On the other hand, Figure~\ref{fig:benoist}~(b) shows that a
non-convex domain may admit a residually convex tessellation if
triangles are allowed. Figure~\ref{fig:benoist}~(a) motivated the
definition of residual convexity because it clearly exhibits one way
in which a non-convex domain may be tessellated by convex polygons.
Both examples are due to Yves Benoist.
\begin{figure}[htbp]
\includegraphics[width=1\linewidth]{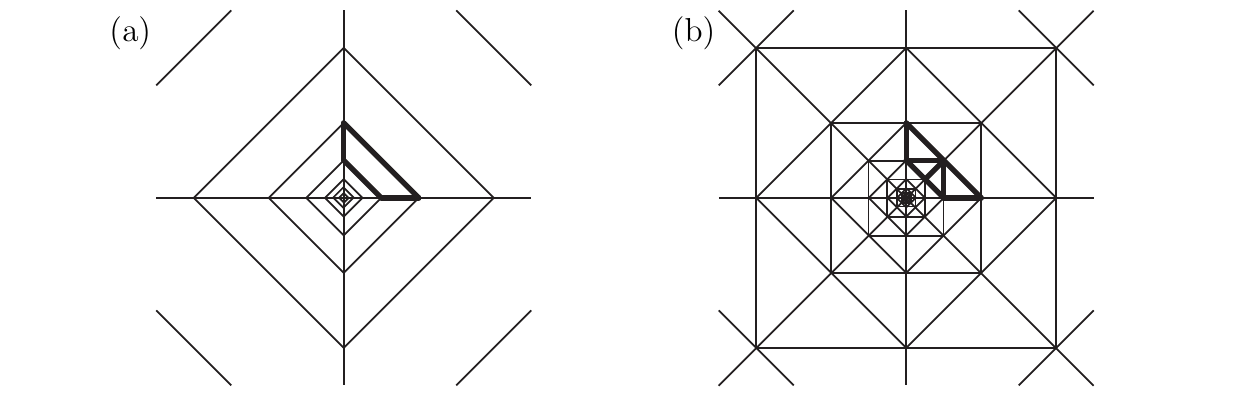}
\caption[8pt]{\footnotesize (a) Given a quadrilateral $P\subset\R^2$
with vertices $(1,0),(2,0),(0,1)$ and $(0,2)$, a tessellation of
$\R^2\setminus\{0\}$ is obtained by taking orbits of $P$ under the
action of the group generated by the homothety by $2$ and the
rotation by $\pi/2$. (b) A residually convex tessellation of
$\R^2\setminus\{0\}$ by triangles is obtained by decomposing the
quadrilateral $P$ in (a) into four triangles.} \label{fig:benoist}
\end{figure}

Our contribution in this paper is to prove the assertions similar to
\textrm{(I)} and \textrm{(II)} above in every dimension -- by
defining appropriate analogues of triangles and polygons with at
least $5$ edges. The former is called a \emph{triangular} polytope
and the latter has a \emph{thick} polytope as its dual. For precise
definitions see Definition~\ref{def:triangular} and
Definition~\ref{def:thick}. As a matter of fact, we prove these
results in a more general context so that they give rise to
convexity criteria for certain real projective structures. From now
on, to the end of the paper, we assume $n\ge2$ except those cases
which are trivially exceptional (like the one in the next
paragraph).

A real projective structure on manifolds is a geometric structure
which is locally modelled on projective geometry $(\R\textup{P}^n,
\textup{Aut}(\R\textup{P}^n))$. If $\Omega\subset\R\textup{P}^n$ is
a convex domain and $\Gamma$ is a discrete subgroup of
$\textup{Aut}(\R\textup{P}^n)$ acting freely and properly
discontinuously on $\Omega$, then the induced real projective
structure on the quotient manifold $\Omega/\Gamma$ is said to be
\emph{convex}. If, moreover, the closure of the convex domain
$\Omega$ does not contain any projective line, then the structure is
called \emph{properly convex}. See Section~\ref{ssec:convex-rpn} for
more details. One of the basic references for real projective
structures is the lecture notes of Goldman \cite{gold-lec}.

Convex real projective structures can be regarded as analogues of
complete Riemannian metrics, and properly convex real projective
structures are expected to share some nice properties with
non-positively curved metrics (see, for example, \cite{benoist} and
\cite{benoist4}). For this reason, given a real projective
structure, one natural question to ask is whether the structure is
(properly) convex. More precisely, let $\{P_i\}$ be a finite family
of convex $n$-dimensional polytopes in $\R\textup{P}^n$. Suppose
that $M$ is a real projective $n$-manifold obtained by gluing
together copies of $P_i$ via projective facet-pairing
transformations. Then there is an associated developing map
$dev:\tilde{M}\to\R\textup{P}^n$ of the universal cover $\tilde{M}$
of $M$, which is a projective isomorphism on each cell of
$\tilde{M}$. One now asks:
\begin{quote}
\textsl{When is the map $dev$ an isomorphism onto a (properly)
convex domain in $\R\textup{P}^n$?}
\end{quote}
The Tits--Vinberg fundamental domain theorem \cite{vinberg} for
discrete linear groups generated by reflections provides a rather
restricted but very constructive solution to this question. Recently
Kapovich \cite{kapovich} proved another convexity theorem when the
$P_i$ are non-compact polyhedra. See Remark~\ref{rem:vin-kap} for a
more detailed discussion. In the present paper, we deal with
complementary cases which are not covered by the aforementioned
results. Our main theorem is as follows (see also
Theorem~\ref{thm:mmain}):

\begin{introthm} \label{introthm:main}
Let $\mathcal{P}$ be a finite family of compact convex
$n$-dimensional polytopes in $\R\textup{P}^n$. Let
$\Phi=\{\phi_\s\in \textup{Aut}(\R\textup{P}^n)\,|\,\s\in\Sigma\}$
be a set of projective facet-pairing transformations for
$\mathcal{P}$ indexed by the collection $\Sigma$ of all facets of
the polytopes in $\mathcal{P}$. Let $M$ be a real projective
$n$-manifold obtained by gluing together the polytopes in
$\mathcal{P}$ by $\Phi$. Assume the following condition:
\begin{quote}
for each facet $\s$ of $P\in\mathcal{P}$, if $\s'$ is a facet of
$P'\in\mathcal{P}$ such that $\phi_\s(\s)=\s'$, then the union
$\phi_\s(P)\cup P'$ is a convex subset of $\R\textup{P}^n$.
\end{quote}
Then the following assertions are true:
\renewcommand{\labelenumi}{\textup{(\Roman{enumi})}}
\begin{enumerate}
\item
If $\mathcal{P}$ contains no triangular polytope, then the
developing map $dev:\tilde{M}\to\R\textup{P}^n$ is an isomorphism
onto a convex domain which is not equal to $\R\textup{P}^n$;
\item
If, in addition, $\mathcal{P}$ contains a polytope $P$ whose dual
$P^*$ is thick, then the map $dev:\tilde{M}\to\R\textup{P}^n$ is an
isomorphism onto a properly convex domain.
\end{enumerate}
\end{introthm}

An interesting related question is whether every convex real
projective structures have convex fundamental domains and how common
residually convex structures are. In \cite{lee2} we provide partial
answer by showing that all properly convex real projective
structures have convex fundamental domains.

\subsection{Convexity}

We sketch our approach to assertion (I) of
Theorem~\ref{introthm:main}. The details are the contents of
Section~\ref{sec:complex} and Section~\ref{sec:convexity}. Let
$X=\tilde{M}$ denote the universal covering space of $M$. We
consider the lift $dev:X\to\S^n$ of the developing map to the sphere
$\S^n$, the two-fold cover of $\R\textup{P}^n$. Regarding $\S^n$
then as the standard Riemannian sphere, we pull back the Riemannian
metric to $X$ via $dev$ so that $X$ is locally isometric to $\S^n$.
Then the simply-connected manifold $X$ becomes a spherical
polyhedral complex.
\renewcommand{\labelenumi}{(\arabic{enumi})}
\begin{enumerate}
\item
In fact, we define such a spherical polyhedral complex $X$ admitting
a developing map $dev$ into $\S^n$ in an abstract way
(\emph{$n$-complex}), so that in general the complex $X$ does not
necessarily admit a cocompact group action (see
Definition~\ref{def:complex}). We call a subset $S\subset X$
\emph{convex} if it is mapped by $dev$ injectively onto a convex
subset of $\S^n$.
\item
We then place on $X$ the \emph{residual convexity} condition, that
is, we require that, for every two $n$-polytopes $P_1$ and $P_2$ in
$X$ sharing a common facet, their union $P_1\cup P_2$ be convex (see
Definition~\ref{def:rconvexity}).
\item
We fix a polytope $P_0$ of $X$ and consider the iterated
\emph{stars} $st^k(P_0)$ of $P_0$ so that they exhaust the whole
complex $X$ (see Definition~\ref{def:star-residue}~(1)). Our plan is
to show \textsl{inductively} that
\begin{quote}
each star $st^k(P_0)$ is convex and its image under $dev$ is not
equal to $\S^n$.
\end{quote}
Then this would imply that $dev:X\to\S^n$ is an isometric embedding
onto a convex proper domain in $\S^n$ (see Theorem~\ref{thm:main}).
\item
Projecting $dev:X\rightarrow\S^n$ down back to $\R\textup{P}^n$ we
get the desired convexity result on the real projective structure on
$M$.
\end{enumerate}
A considerable portion of the present paper is devoted to step (3)
of the above plan. We now explain how the induction argument goes:
\renewcommand{\labelenumi}{(\roman{enumi})}
\begin{enumerate}
\item
It turns out that the residual convexity establishes the base step
of the induction (see Lemma~\ref{lem:residue} (1) and
Lemma~\ref{lem:stars} (1)).
\item
We assume that the $k$-th star $st^k(P_0)$ is convex and its image
under $dev$ is not equal to $\S^n$. Then it is rather easy to show
that the $(k+1)$-th star $st^{k+1}(P_0)$ is mapped injectively onto
a topological ball (\emph{$n$-polyball}) in $\S^n$ (see
Lemma~\ref{lem:stars} (2) and Definition~\ref{def:polyball}).
\item
We next want to show that the star $st^{k+1}(P_0)$ is locally
convex. Because of its polyhedral structure, the local convexity of
$st^{k+1}(P_0)$ can be drawn from its local convexity near
codimension-$2$ cells (\emph{ridges}) in the boundary (see
Lemma~\ref{lem:ridge}).
\item
Let $e$ be a codimension-$2$ cell in the boundary of the star
$st^{k+1}(P_0)$. The local geometry of $st^{k+1}(P_0)$ near $e$ is
determined by the union $U(e)$ of $n$-cells in $X$ which contain $e$
and which intersect $st^k(P_0)$. Thus we need to find conditions
which imply that the union $U(e)$ is convex. Interestingly, there is
a \textsl{local} condition for this.
\item
Indeed, we consider a small neighborhood $res(e)$ (\emph{residue})
of $e$ which consists of those $n$-cells in $X$ which contain $e$
(see Definition~\ref{def:star-residue} (2)). Residual convexity
implies that $res(e)$ is convex (see Lemma~\ref{lem:residue} (3)).
Because the star $st^k(P_0)$ is also assumed to be convex and
because $st^k(P_0)$ and $res(e)$ intersect along their boundaries,
their intersection $F:=st^k(P_0)\cap res(e)$ is a convex subset in
the boundary of $res(e)$. Then the union $U(e)$ can be described as
the union $U(e,F)$ of $n$-cells in $res(e)$ which intersect $F$.
\item
The condition, which we call \emph{strong residual convexity},
requires that, for all $e$, the set $U(e,F)$ be always convex
regardless of convex subsets $F$ in the boundary of $res(e)$ (see
Definition~\ref{def:good} and Definition~\ref{def:strong}).
Figure~\ref{fig:strong} illustrates the case where strong residual
convexity fails. In conclusion, under the assumption of strong
residual convexity, we can show that the star $st^{k+1}(P_0)$ is
locally convex near codimension-$2$ cells in its boundary (see
Lemma~\ref{lem:main}).
\begin{figure}[htbp]
\includegraphics[width=1\linewidth]{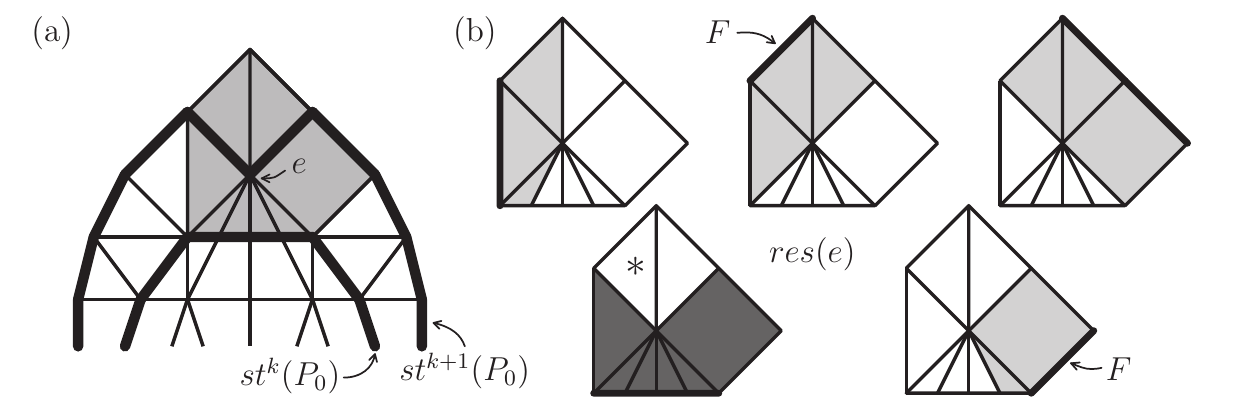}
\caption{\footnotesize Strong residual convexity. (a) A
codimension-$2$ cell $e$ is in the boundary of $st^{k+1}(P_0)$. The
union of $n$-cells containing $e$ forms a convex neighborhood
$res(e)$ of $e$, which intersects the convex set $st^k(P_0)$ along
its boundary. (b) The set $res(e)$ has five maximal convex subsets
$F$ in its boundary. For one of such $F$, the union of $n$-cells of
$res(e)$ intersecting $F$ is not convex. The corresponding picture
is marked by (*).} \label{fig:strong}
\end{figure}
\item
Finally, once the local convexity is established, we may regard the
star $st^{k+1}(P_0)$ as an Alexandrov space of curvature $\ge1$ and
then deduce its global convexity using a well-known local-to-global
theorem for such spaces (see Corollary~\ref{cor:locglo}). All
induction steps are complete.
\end{enumerate}
To summarize, we have the following convexity theorem:
\begin{introthm}\label{intro:main}
Let $X$ be an $n$-complex. If $X$ is strongly residually convex,
then $X$ is isometric to a convex proper domain in $\S^n$. In
particular, $X$ is contractible.
\end{introthm}

As can be seen in steps (iii)-(vi) above, the
\textsl{codimension-$2$ phenomena} in polyhedral complexes enables
us to go from dimension $2$ to arbitrary dimensions. This is a
rather common trick which can be found, for example, in the proof of
the Poincar\'e fundamental polyhedron theorem for constant curvature
spaces (see, for example, \cite{ep} and \cite{rat}). However, we
find it worthwhile to develop this trick into a form which is
suitable for our present purpose. Hence the most of
Section~\ref{sec:prelim} is devoted to the study of geometric links
of faces of various dimensions in convex polytopes.

Although strong residual convexity is entirely a local condition,
for practical reasons, it is desirable to have simple combinatorial
conditions under which residual convexity becomes strong residual
convexity. Observe that triangles caused the failure of strong
residual convexity in Figure~\ref{fig:strong}. See also
Figure~\ref{fig:notriangle}. Using the codimension-$2$ phenomena
once again, we define \emph{triangular} polytopes and show that
without presence of triangular polytopes residual convexity implies
strong residual convexity (see Theorem~\ref{thm:notriangle}).
Combining this result with Theorem~\ref{intro:main} we obtain the
following corollary, which again implies assertion \textrm{(I)} of
Theorem~\ref{introthm:main}.

\begin{introcor} \label{intro:cor}
Let $X$ be a residually convex complex. If $X$ contains no
triangular polytopes, then $X$ is isometric to a convex domain which
is not $\S^n$.
\end{introcor}

\subsection{Proper convexity}

We now outline our approach to assertion \textrm{(II)} of
Theorem~\ref{introthm:main}. The details are explained in
Section~\ref{sec:pconvexity}. The starting point is the above
Corollary~\ref{intro:cor}. That is, we assume that our complex $X$
is residually convex and contains no triangular polytopes. Then $X$
is isometric to a convex domain in $\S^n$. Thus from now on we
regard $X$ as a convex subset of $\S^n$ and find conditions implying
proper convexity of $X$.

Our eventual plan is to find $n+1$ supporting hyperplanes of $X$
that are in general position. Then $X$ is contained in the
$n$-simplex which is determined by these hyperplanes. Because
$n$-simplices are properly convex, the conclusion then follows.
Fortunately, there is a natural way to find supporting hyperplanes
of $X$ provided that $X$ contains no triangular polytope. Thus we
need to find further conditions under which there are $n+1$ such in
general position.

For example, if $X$ is $2$-dimensional and contains no triangle, all
polygons in $X$ have at least four edges and this enables us to
construct the following objects in $X$. We fix a polygon $Q_0$ in
$X$. Given an edge $e_0$ of $Q_0$, consider the polygon $Q_1$ that
is adjacent to $Q_0$ along the common edge $e_0$. Then we can choose
an edge $e_1$ of $Q_1$ which is disjoint from $e_0$. We then
consider the polygon $Q_2$ adjacent to $Q_1$ along $e_1$. Choose an
edge $e_2$ of $Q_2$ which is disjoint from $e_1$, and so on. This
process defines an infinite sequence (\emph{directed gallery}) of
adjacent polygons in $X$ (see Figure~\ref{fig:pentagon}~(b) and
Definition~\ref{def:gal}). One can then show that the limit of the
lines spanned by the edges $e_i$ is a supporting line to $X$. Now,
if the polygon $Q_0$ is, say, a pentagon then we have five such
supporting lines constructed from the edges of $Q_0$ as above. It is
easy to see that two supporting lines coming from two nearby edges
of $Q_0$ may coincide but those coming from disjoint edges of $Q_0$
never coincide. Because $5=2+2+1$, this implies that there are at
least three supporting lines of $X$ which are in general position so
that they bound a triangle (see Figure~\ref{fig:pentagon}~(c)).

We now explain how the previous arguments in dimension $2$ can be
generalized to higher dimensions:
\renewcommand{\labelenumi}{(\alph{enumi})}
\begin{enumerate}
\item
To be able to define directed galleries, we need the analogues of
polygons with at least four edges. For this, we re-interpret
triangles and define \emph{cone-like} polytopes (see
Definition~\ref{def:cone-like}). If none of the polytopes in $X$ is
cone-like then we can define directed galleries in $X$. It turns out
that non-triangular polytopes are not cone-like (see
Lemma~\ref{lem:cone-like}).
\item
Fix a polytope $Q$ in $X$. Each directed gallery associated to a
facet $\s$ of $Q$ defines a supporting hyperplane $H_X(\s)$ of $X$.
Because every $n$-polytope has at least $n+1$ facets, we have at
least $n+1$ such supporting hyperplanes.
\item
Such simple counting as $5=2+2+1$ above does not work in higher
dimensions, where both combinatorial and geometric arguments are
necessary. To deal with the arrangement of supporting hyperplanes,
we consider the dual $Q^*$ of $Q$ and points $x(\s)$ dual to the
halfspaces $H_X(\s)^+$ which contain $X$ and which are bounded by
the supporting hyperplanes $H_X(\s)$. On the other hand, the
vertices $\s^*$ of $Q^*$ are dual to the halfspaces
$\langle\s\rangle^+$ which contain $Q$ and which are bounded by the
hyperplanes $\langle\s\rangle$ spanned by facets $\s$ of $Q$.
\item
Each hyperplane $H_X(\s)$ associated to a facet $\s$ of $Q$ has some
restriction on its location (see Lemma~\ref{lem:restriction1}). We
translate this restriction in terms of duality to obtain a subset
(\emph{pavilion}) of $Q^*$ associated to the vertex $\s^*$, to which
the point $x(\s)$ must belong (see Definition~\ref{def:pav} and
Lemma~\ref{lem:restriction2}).
\item
Finally, we prove that if $Q^*$ is \emph{thick} then there always
exist $n+1$ such points $x(\s)$ in general position, which again
implies that there always exist $n+1$ supporting hyperplanes
$H_X(\s)$ of $X$ in general position (see Lemma~\ref{lem:thin}).
\end{enumerate}
In summary, we have the following theorem (see
Theorem~\ref{thm:pconvexity}) which implies the assertion
\textrm{(II)} of Theorem~\ref{introthm:main}:
\begin{introthm}
Let $X\subset\S^n$ be a residually convex $n$-complex such that none
of the $n$-cells of $X$ are triangular. If $X$ has an $n$-cell $Q$
whose dual $Q^*$ is thick, then $X$ is a properly convex domain in
$\S^n$.
\end{introthm}

In the final Section~\ref{sec:rpn} we discuss real projective
structures in more detail and explain how all these results are
applied to give convexity theorem for certain real projective
structures.

\subsection{Remark}
It should be noted that we introduce metric to prove
Theorem~\ref{introthm:main}, which does not involve any
metric-dependent notion. There are two main reasons for using metric
in our discussion:
\begin{itemize}
\item
When we consider links of polytopes and argue inductively, we can
embed links of various dimension in a single space $\S^n$ so that
our presentation gains more convenience and geometric flavor.
However, this is not an essential ingredient in our proof and there
is a more natural way of defining links without using metric (see
Remark~\ref{rem:metric}).
\item
We can use a local-to-global theorem for Alexandrov spaces of
curvature bounded below (see Theorem~\ref{thm:locglo}). We do not
know how to draw global convexity of spherical domains from their
local convexity without using this theorem.
\end{itemize}

\subsection*{Acknowledgements} My advisor Misha Kapovich recommended
me to investigate the property of residual convexity. I am grateful
to him for this and I deeply appreciate his encouragement and
patience during my work. I also thank Yves Benoist and Damian Osajda
for helpful discussions. During this work I was partially supported
by the NSF grants DMS-04-05180 and DMS-05-54349.

\section{Preliminaries} \label{sec:prelim}

Let $\R^n$ be the $n$-dimensional Euclidean vector space. We denote
the origin by $o$ and the standard inner product by
$\langle\;,\;\rangle$. Given a linear subspace $L$ its orthogonal
complement is denoted $L^\bot$. For two subsets $S_1$ and $S_2$
their sum $S_1+S_2$ is the set of all points $x_1+x_2$ for $x_1\in
S_1$ and $x_2\in S_2$.

Let $S$ be a subset of $\R^n$ whose closure $\overline{S}$ contains
the origin $o$. The smallest linear subspace containing $S$ is
denoted $L(S)$. The \emph{(linear) dimension} of $S$ is defined to
be the dimension of this subspace. We say that $S$ is \emph{open} if
it is open relative to $L(S)$. A point $x\in S$ is called an
\emph{interior} (resp. \emph{boundary}) point of $S$ if $x$ is an
interior (resp. boundary) point of $S$ relative to $L(S)$.

\subsection{Convex cones}

A subset $S\subset\R^n$ is said to be \emph{convex} if for every
$x,y\in S$ and for every $a\ge0, b\ge0$ such that $a+b=1$ the point
$ax+by$ is in $S$, that is, the affine line segment joining $x$ and
$y$ is in $S$. One can show that if $S$ is convex then its closure
$\overline{S}$ is also convex. The \emph{convex hull}
$\textup{conv}(S)$ of a subset $S$ is the smallest convex subset
containing $S$. A \emph{cone} $C$ is a subset of $\R^n$ such that if
$x\in C$ and $a>0$ then $ax\in C$. Thus cones are invariant under
positive homotheties of $\R^n$. Note that for any cone $C$ its
closure $\overline{C}$ necessarily contains the origin $o$.

A \emph{convex cone} is a cone which is convex. Linear subspaces and
halfspaces bounded by codimension-$1$ linear subspaces are convex
cones; these examples contain a complete affine line. A convex cone
is called \emph{line-free} if it contains no complete affine line.
Given a convex cone $C$ we denote by $l(C)$ the largest linear
subspace contained in $\overline{C}$. The following lemma says that
a closed convex cone decomposes into a linear part and a line-free
part; compare with \cite{gold-lec} and \cite{grun}. See also
Figure~\ref{fig:decomp}(a).

\begin{lemma}[Decomposition Theorem] \label{lem:decomp}
Let $C$ be a convex cone in $\R^n$. Then $l(C)=\{o\}$ if and only if
$\overline{C}$ is line-free. If $l(C)\neq\{o\}$ then $\overline{C}$
decomposes into
$$
\overline{C}=(\overline{C}\cap l(C)^\bot)+l(C)
$$
and $\overline{C}\cap l(C)^\bot$ is a line-free convex cone, where
$l(C)^\bot$ denotes the orthogonal complement of $l(C)$.
\end{lemma}

\begin{proof} Let $x$ and $y$ be two points in $\overline{C}$. We first
claim that $\overline{C}$ contains the complete affine line
$\{x+tz\,|\,t\in\R\}$ passing through $x$ in the direction of $z$ if
and only if it contains the parallel line $\{y+tz\,|\,t\in\R\}$
passing through $y$. Suppose first that $\overline{C}$ contains the
line $\{x+tz\,|\,t\in\R\}$. Then for any $s>0$ and $t\in\R$, the
point
$$
y_{s,t}=\frac{s}{s+1}y+\frac{1}{s+1}(x+stz)
$$
is on the affine segment joining $y$ and $x+stz$. Because
$\overline{C}$ is convex the point $y_{s,t}$ is in $\overline{C}$.
As $s$ goes to infinity, however, $y_{s,t}$ converges to $y+tz$.
Since $\overline{C}$ is closed, this shows that $\overline{C}$
contains the line $\{y+tz\,|\,t\in\R\}$. Since $x$ and $y$ play the
equivalent roles, this completes the proof of the claim.

Recall that $\overline{C}$ contains the origin $o$. Then the above
claim says that $\overline{C}$ contains a complete affine line if
and only if it contains a $1$-dimensional subspace. Therefore,
$l(C)=\{o\}$ if and only if $\overline{C}$ is line-free.

So from now on we suppose that $l(C)\neq\{o\}$. Because
$l(C)\subset\overline{C}$ and any translate $x+l(C)$ of $l(C)$
intersects $l(C)^\bot$, it follows from the above claim that
$\overline{C}$ decomposes into $\overline{C}=(\overline{C}\cap
l(C)^\bot)+l(C)$. Since both $\overline{C}$ and $l(C)^\bot$ are
convex cones, their intersection $\overline{C}\cap l(C)^\bot$ is
also a convex cone. Suppose by way of contradiction that
$\overline{C}\cap l(C)^\bot$ contains a complete affine line. The
above claim then shows that it also contains a $1$-dimensional
subspace $l$. But the subspace $l+l(C)$ properly contains $l(C)$ and
is contained in $\overline{C}$; this is contradictory to the
definition of $l(C)$. The proof of lemma is complete.
\end{proof}

\begin{remark} \label{rem:metric}
We can avoid using metric $\langle\;,\;\rangle$ and state
Lemma~\ref{lem:decomp} in terms of quotient space instead of
orthogonal complement. Namely, let $\pi_{l(C)}:\R^n\to\R^n/l(C)$ be
the natural projection onto $\R^n/l(C)$. Then
$\pi_{l(C)}(\overline{C})$ is a line-free convex cone in $\R^n/l(C)$
such that $\overline{C}=\pi_{l(C)}^{-1}[\pi_{l(C)}(\overline{C})]$.
We may consider $\pi_{l(C)}(\overline{C})$ as the line-free part of
$\overline{C}$ and use this to define links of polyhedral cones and
polytopes in the following discussion. While we can proceed in this
more natural way, we prefer using metric for the sake of
presentational convenience.
\end{remark}

A \emph{hyperplane} is an $(n-1)$-dimensional linear subspace of
$\R^n$. Let $C$ be a convex cone. We say that a hyperplane $H$
\emph{supports} $C$ if $C$ is contained in one of the closed
halfspaces bounded by $H$; this halfspace is denoted by $H^+$ (and
the other one by $H^-$) and is also said to support $C$. In fact, it
can be shown that if $C\neq\R^n$ then $C$ is contained in some
halfspace of $\R^n$ (see for example \cite{fenchel}). A non-empty
subset $f\subsetneq C$ is called a \emph{face} of $C$ if there is a
supporting hyperplane $H$ of $C$ such that $f=C\cap H$. Obviously,
faces of $C$ are also convex cones.

\subsection{Polyhedral cones} \label{polycone}

A subset $P\subset\R^n$ is called a \emph{polyhedral cone} if it is
the intersection of a finite family of closed halfspaces of $\R^n$.
Clearly, polyhedral cones are closed convex cones. A polyhedral cone
$P$ is \emph{polytopal} if it is line-free, that is, $l(P)=\{o\}$.

Let $P$ be a polyhedral cone in $\R^n$. It is known that if $f$ is a
face of $P$ then faces of $f$ are also faces of $P$. A maximal face
of $P$ is called a \emph{facet} of $P$. A \emph{ridge} of $P$ is a
facet of a facet of $P$. Let $P=\bigcap_{i=1}^m H_i^+$ where the
$H_i^+$ are halfspaces bounded by hyperplanes $H_i$. We further
assume that the family $\{H_i^+\}$ is \emph{irredundant}, that is,
$$
\bigcap_{j\neq i} H_j^+ \neq P
$$
for each $i=1,2,\ldots,m$. The irredundancy condition implies the
following properties of faces of $P$ (see \cite{grun}):
\begin{itemize}
\item
If $P$ is $n$-dimensional, a facet of $P$ is of the form $P\cap H_i$
for some $i$;
\item
The boundary of $P$ is the union of all facets of $P$;
\item
Each ridge of $P$ is a non-empty intersection of two facets of $P$;
\item
Every face of $P$ is a non-empty intersection of facets of $P$.
\end{itemize}
Thus the number of faces of $P$ is finite. If $P$ is
$n'$-dimensional then its facets are $(n'-1)$-dimensional and ridges
are $(n'-2)$-dimensional.

\subsection{Links in polyhedral cones}

Let $P=\bigcap_{i=1}^m H_i^+$ be a polyhedral cone in $\R^n$. Let
$f$ be a face of $P$. If $P$ is $n$-dimensional then we may assume
without loss of generality that $f$ is the intersection of facets
$P\cap H_1,\ldots,P\cap H_{m_f}$ of $P$ for some $m_f<m$, that is,
$$
f=(P\cap H_1)\cap\cdots\cap(P\cap H_{m_f})=P\cap(H_1\cap\cdots\cap
H_{m_f}).
$$
Because any sufficiently small neighborhood of an interior point of
$f$ intersects only those hyperplanes $H_i$ which contain $f$, the
local geometry of $P$ near an interior point of $f$ is the same as
the local geometry near the origin $o$ of the polyhedral cone
determined by the corresponding halfspaces $H_i^+$. We denote this
polyhedral cone by
$$
P_f=H_1^+\cap\cdots\cap H_{m_f}^+.
$$
By Lemma~\ref{lem:decomp}, the polyhedral cone $P_f$ decomposes into
$$
(P_f\cap l(P_f)^\bot)+l(P_f).
$$
However, the linear part $l(P_f)$ is just the intersection
$H_1\cap\cdots\cap H_{m_f}$, which is again equal to the smallest
linear subspace $L(f)$ containing $f$. Thus we have
$$
P_f=(P_f\cap L(f)^\bot)+L(f).
$$
Now the \emph{link} $\L(f;P)$ of $f$ in $P$ is defined to be the
line-free part of $P_f$:
$$
\L(f;P)=P_f\cap L(f)^\bot=\bigcap_{i=1}^{m_f} (H_i^+\cap L(f)^\bot).
$$
See Figure~\ref{fig:decomp} (b).
\begin{figure}[htbp]
\includegraphics[width=1\linewidth]{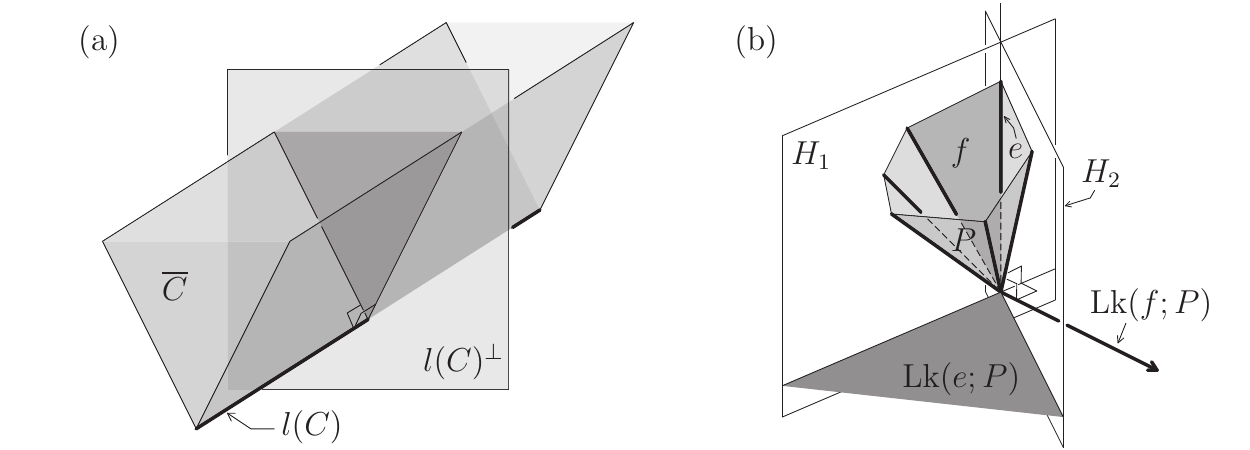}
\caption{\footnotesize (a) Illustration of Lemma~\ref{lem:decomp}
(Decomposition Theorem). (b) Links $\L(e;P)$ and $\L(f;P)$ in a
polytopal cone $P$.} \label{fig:decomp}
\end{figure}
If $f$ has dimension $m$ then $L(f)$ is $m$-dimensional and
$L(f)^\bot$ is $(n-m)$-dimensional. Because $P_f$ has full-dimension
in $\R^n$, $P_f\cap L(f)^\bot$ is also full-dimensional in
$L(f)^\bot$. It follows that the link $\L(f;P)$ is an
$(n-m)$-dimensional \textsl{polytopal} cone in
$L(f)^\bot\subset\R^n$ with its defining halfspaces being $H_i^+\cap
L(f)^\bot$.

We defined the link $\L(f;P)$ under the assumption that $P$ is an
$n$-dimensional polyhedron in $\R^n$. If $P$ is $n'$-dimensional
with $n'<n$, however, we just consider the smallest linear subspace
$L(P)$ containing $P$ and define the link $\L(f;P)$ with respect to
$L(P)$ in the same manner as above. Thus if $f$ is $m$-dimensional,
its link $\L(f;P)$ is an $(n'-m)$-dimensional polytopal cone in
$L(P)\cap L(f)^\bot\subset\R^n$.

Let $P$ be an $n$-dimensional polyhedral cone in $\R^n$. Let $f$ be
a face of $P$ and $e$ a face of $f$. We define a subset $f_{(e;P)}$
of the link $\L(e;P)$ as:
$$
f_{(e;P)}=\L(e;P)\cap L(f).
$$
The lemma below says that $f_{(e;P)}$ is a face of the polytopal
cone $\L(e;P)$, whose link in $\L(e;P)$ is equal to the link
$\L(f;P)$. Thus the link $\L(e;P)$ of $e$ has all the information
about the links $\L(f;P)$ of those faces $f$ which contain $e$; this
fact enables us to use inductive arguments on links later on.

\begin{lemma} \label{links1}
Let $P$ be an $n$-dimensional polyhedral cone in $\R^n$. Let $f$ be
a face of $P$ and $e$ a face of $f$. Then $f_{(e;P)}$ is a face of
the polytopal cone $\L(e;P)$. If $f$ is a facet of $P$ then
$f_{(e;P)}$ is also a facet of $\L(e;P)$. Furthermore, we have the
following identity between the two links involved:
$$
\L(f;P)=\L[f_{(e;P)};\L(e;P)].
$$
\end{lemma}

\begin{proof}
We write $P=\bigcap_{i=1}^m H_i^+$ for an irredundant family
$\{H_i^+\}$ of halfspaces of $\R^n$ bounded by $H_i$. We may assume
that for some $m_f<m_e<m$ the faces $f$ and $e$ are expressed as
\begin{align*}
f&=P\cap(H_1\cap\cdots\cap H_{m_f})\\
e&=P\cap(H_1\cap\cdots\cap H_{m_f}\cap H_{m_f+1}\cap\cdots\cap
H_{m_e}).
\end{align*}
If we set, as before,
\begin{align*}
P_f&=H_1^+\cap\cdots\cap H_{m_f}^+\\
P_e&=H_1^+\cap\cdots\cap H_{m_f}^+\cap H_{m_f+1}^+\cap\cdots\cap
H_{m_e}^+,
\end{align*}
then the links of $f$ and $e$ are by definition
\begin{align*}
\L(f;P)&=P_f\cap L(f)^\bot\\
\L(e;P)&=P_e\cap L(e)^\bot=\bigcap_{i=1}^{m_e} (H_i^+\cap
L(e)^\bot).
\end{align*}
Because $L(f)=H_1\cap\cdots\cap H_{m_f}$ and $\L(e;P)\subset
L(e)^\bot$, we then have
\begin{align*}
f_{(e;P)}&=\L(e;P)\cap L(f)\\
&=\L(e;P)\cap(H_1\cap\cdots\cap H_{m_f})\\
&=\L(e;P)\cap[(H_1\cap L(e)^\bot)\cap\cdots\cap(H_{m_f}\cap
L(e)^\bot)].
\end{align*}
Since $m_f<m_e$ and the defining halfspaces of $\L(e;P)$ are
$H_i^+\cap L(e)^\bot$ $(1\le i\le m_e)$, this shows that $f_{(e;P)}$
is a face of the polytopal cone $\L(e;P)$. If $f$ is a facet of $P$
then $m_f=1$ and $f=P\cap H_1$. Therefore,
$f_{(e;P)}=\L(e;P)\cap(H_1\cap L(e)^\bot)$ is a facet of $\L(e;P)$.

To see the claimed equality we first note that, because
$\L(e;P)\subset L(e)^\bot$ has non-empty interior in $L(e)^\bot$,
\begin{align*}
L(f_{(e;P)})=L[\L(e;P)\cap L(f)]=L(e)^\bot\cap L(f).
\end{align*}
Because $e\subset f$ and hence $L(e)\subset L(f)$, we then have
\begin{align*}
L(e)^\bot\cap
L(f_{(e;P)})^\bot=L(e)^\bot\cap(L(e)+L(f)^\bot)=L(f)^\bot.
\end{align*}
Finally, unraveling all the definitions, we see that
\begin{align*}
\L[f_{(e;P)};\L(e;P)]&=\L(e;P)_{f_{(e;P)}}\cap L(f_{(e;P)})^\bot\\
&=[(H_1^+\cap L(e)^\bot)\cap\cdots\cap(H_{m_f}^+\cap L(e)^\bot)]\cap L(f_
{(e;P)})^\bot\\
&=(H_1^+\cap\cdots\cap H_{m_f}^+)\cap L(e)^\bot\cap L(f_{(e;P)})^\bot\\
&=P_f\cap L(f)^\bot\\
&=\L(f;P).\qedhere
\end{align*}
\end{proof}

\subsection{Spherical polytopes}

Let $\S^n$ be the unit sphere in $\R^{n+1}$. To any subset
$S\subset\S^n$ we associate the \emph{cone} $\c_S$ \emph{over} $S$
defined by
$$
\c_S=\{ax\in\R^{n+1}\,|\,x\in S, a\ge0\}.
$$
For a subset $S\subset\S^n$ and a cone $C\subset\R^{n+1}$, it is
clear that
$$
\c_S\cap\S^n=S\;\;\text{and}\;\;\c_{C\cap\S^n}=C\cup\{o\}.
$$
A subset $L\subset\S^n$ is an \emph{$m$-plane} provided that the
cone $\c_L$ over $L$ is an $(m+1)$-dimensional linear subspace of
$\R^{n+1}$. The orthogonal complement $L^\bot$ of an $m$-plane $L$
is defined to be $(\c_L)^\bot\cap\S^n$.

Let $S$ be a subset of $\S^n$. The smallest $m$-plane containing $S$
is denoted $L(S)$ and is clearly equal to $L(\c_S)\cap\S^n$. The
\emph{dimension} of $S$ is defined to be the dimension of this
plane. We call $S$ \emph{open} if it is open relative to $L(S)$.
Likewise, a point $x\in S$ is called an \emph{interior} (resp.
\emph{boundary}) point of $S$ if $x$ is an interior (resp. boundary)
point of $S$ relative to $L(S)$. We also denote by $S^\circ$ the set
of interior points of $S$.

A subset $S\subset\S^n$ is \emph{convex} (resp. \emph{properly
convex}) if the cone $\c_S$ over $S$ is a convex cone (resp.
line-free convex cone). It is clear that $S\subset\S^n$ is convex if
and only if for any two points in $S$ the (spherical) geodesic
connecting them is in $S$. A subset $S\subset\S^n$ is \emph{locally
convex} if every point of $S$ has a neighborhood in $S$ which is a
convex subset of $\S^n$. The \emph{convex hull} $\textup{conv}(S)$
of a subset $S$ is the smallest convex subset containing $S$.
Finally, a subset $S\subset\S^n$ is a \emph{noun} if the cone $\c_S$
over $S$ is a noun in $\R^{n+1}$, where the noun stands for
\emph{hyperplane, halfspace, support} or \emph{face}. Note that if
$S\neq\S^n$ is convex then $\c_S\neq\R^{n+1}$ is a convex cone and
is contained in a halfspace of $\R^{n+1}$. Thus every convex subset
$S$ not equal to $\S^n$ is contained in a halfspace of $\S^n$ and
hence has diameter at most $\pi$.

A subset $P\subset\S^n$ is a \emph{polyhedron} (resp.
\emph{polytope}) if the cone $\c_P$ over $P$ is a polyhedral cone
(resp. polytopal cone) in $\R^{n+1}$. If a polyhedron $P$ has
dimension $m$ we call $P$ an $m$-polyhedron and similarly for
polytopes. A maximal face of $P$ is called a \emph{facet} of $P$. A
\emph{ridge} of $P$ is a facet of a facet of $P$. A \emph{vertex}
(resp. \emph{edge}) of $P$ is a $0$-dimensional (resp.
$1$-dimensional) face of $P$. Let $P=\bigcap_{i=1}^m H_i^+$ where
the $H_i^+$ are halfspaces bounded by hyperplanes $H_i$, that is,
$H_i^+=(\c_{H_i})^+\cap\S^n$. Under the same irredundancy condition
on the family $\{H_i^+\}$ as in Section~\ref{polycone}, the same
properties of faces of $P$ as listed therein hold.

Let $P\subset\S^n$ be a polyhedron and $f$ a face of $P$. The
\emph{link} $\L(f;P)$ of $f$ in $P$ is by definition
$$
\L(f;P)=\L(\c_f;\c_P)\cap\S^n.
$$
See Figure~\ref{fig:links}.
\begin{figure}[htbp]
\includegraphics[width=1\linewidth]{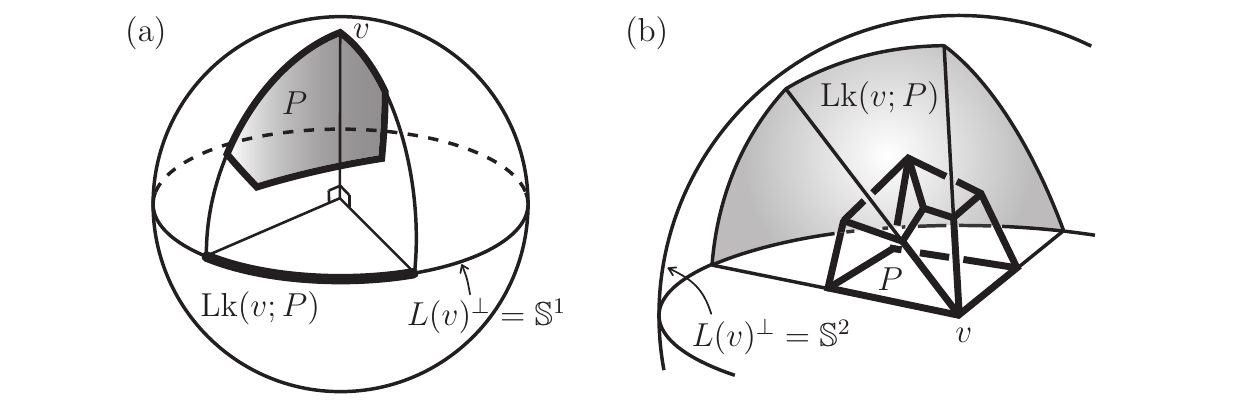}
\caption{\footnotesize Links $\L(v;P)$ of $v$ in $P$ are drawn in
(a) and (b). The ambient space of (b) is $\S^3$.} \label{fig:links}
\end{figure}
Because $\L(\c_f;\c_P)$ is a polytopal cone, the link $\L(f;P)$ is a
polytope in $\S^n$. If $P$ is an $n$-polyhedron and $f$ is an
$m$-face then the link $\L(f;P)$ is an $(n-m-1)$-polytope. Let $e$
be a face of $f$ and define a subset $f_{(e;P)}$ of $\L(e;P)$ by
$$
f_{(e;P)}=\L(e;P)\cap L(f).
$$
It then follows from Lemma~\ref{links1} that $f_{(e;P)}$ is a face
of the polytope $\L(e;P)$ and the following identity holds between
the two links involved:
\begin{align} \label{links2}
\L(f;P)=\L(f_{(e;P)};\L(e;P)).
\end{align}

\subsection{Duality} \label{ssec:dual}

Let $\R_n$ be the dual vector space $(\R^n)^*=\textup{Hom}(\R^n,\R)$
of $\R^n$. It is equipped with the standard inner product coming
from that of $\R^n$. Denote by $\S_n$ the unit sphere in $\R_n$.

Let $C$ be a cone in $\R^n$. The \emph{dual cone} $C^*$ of $C$ is
defined by
$$
C^*=\{u\in\R_n\,|\,u(x)\le0\;\text{for all}\;x\in C\}.
$$
It is easy to see that $C^*$ is a closed convex cone in $\R_n$. If
$L$ is an $m$-dimensional linear subspace of $\R^n$ then $L^*$ is an
$(n-m)$-dimensional linear subspace of $\R_n$. If $H^+$ is a
halfspace bounded by a hyperplane $H$ then $(H^+)^*$ is a ray in
$\R_n$. We have the following well-known facts (compare with
\cite{grun} and \cite{fenchel}):
\begin{itemize}
\item
If $C$ is a closed convex cone then $C^{**}=C$ (under the natural
identification $(\R_n)^*=\R^n$) and
\begin{align*}
\dim L(C^*) +\dim l(C)&=n;\\
\dim L(C)+\dim l(C^*)&=n.
\end{align*}
\item
If $C$ and $D$ are closed convex cones then
$$
(C\cap D)^*=\textup{conv}(C^*\cup D^*).
$$
\item
If $P$ is a polyhedral cone then so too is $P^*$.
\item
If $P$ is an $n$-dimensional polytopal cone then so too is $P^*$.
\end{itemize}

Let $S$ be a subset of $\S^n$. The \emph{dual} $S^*$ of $S$ is
defined by
$$
S^*=(\c_S)^*\cap\S_n.
$$
Thus the dual $S^*$ of $S$ is always a closed convex subset of
$\S_n$. If $L\subset\S^n$ is an $m$-plane then $L^*$ is an
$(n-m-1)$-plane. In particular, the dual of a hyperplane $H$ is a
pair $\{\pm v\}=\S_0$ of antipodal points. The dual of a halfspace
is a single point; if $(H^+)^*=v$ then $(H^-)^*=-v$. The analogous
properties for cones as listed above also hold for subsets of
$\S^n$. In particular, if $P\subset\S^n$ is an $n$-polytope then so
too is its dual $P^*$; if $P$ is expressed as
$$
P=\bigcap_{i=1}^m H_i^+,
$$
then
$$
P^*=\left[\bigcap_{i=1}^m
H_i^+\right]^*=\textup{conv}\left[\bigcup_{i=1}^m
(H_i^+)^*\right]=\textup{conv}\{v_1,v_1,\ldots,v_m\},
$$
where each $v_i=(H_i^+)^*$ becomes a vertex of the dual polytope
$P^*$.

\subsection{Alexandrov spaces of curvature bounded below}

The main reference for this subsection is \cite{bgp}. Fix a real
number $\kappa$. Let $M_\kappa^n$ be the $n$-dimensional complete
simply-connected Riemannian manifold of constant curvature $\kappa$,
and denote $D_\kappa=\pi/\sqrt{\kappa}$ for $\kappa>0$ and
$D_\kappa=\infty$ for $\kappa\le0$. Thus, for example, we have
$M_1^n=\mathbb{S}^n$ and $D_1=\pi$. We denote by $d$ the induced
path metric on $M_\kappa^n$.

Let $X$ be a metric space. Given three points $p,q,r\in X$
satisfying
$$
d(p,q)+d(q,r)+d(r,p)<2D_\kappa,
$$
there is a comparison triangle $\triangle(\bar{p},\bar{q},\bar{r})$
in $M_\kappa^2$, namely, three points $\bar{p},\bar{q},\bar{r}\in
M_\kappa^2$ such that
$$
d(\bar{p},\bar{q})=d(p,q),\;\; d(\bar{q},\bar{r})=d(q,r),\;\;
d(\bar{r},\bar{p})=d(r,p).
$$
We define $\bar{\angle}pqr$ to be the angle at the vertex $\bar{q}$
of the triangle $\triangle(\bar{p},\bar{q},\bar{r})$.

Let $X$ be a path metric space, that is, a metric space where the
distance $d$ between each pair of points is equal to the infimum of
the length of rectifiable curves joining them. Then $X$ is said to
be $\textup{Alex}(\kappa)$ provided that for any four distinct
points $b,c,d$ and $a$ in $X$ we have the inequality
$$
\bar{\angle}bac + \bar{\angle}cad + \bar{\angle}dab \le 2\pi.
$$
(If $X$ is a $1$-dimensional manifold and $\kappa>0$, then we
require in addition that its diameter be at most $D_\kappa$.) The
path metric space $X$ is said to be \emph{locally
$\textup{Alex}(\kappa)$}, or more commonly, \emph{an Alexandrov
space of curvature $\ge\kappa$}, if each point $x\in X$ has a
neighborhood $U_x$ which is $\textup{Alex}(\kappa)$.

Examples of locally $\textup{Alex}(\kappa)$ spaces include
Riemannian manifolds without boundary or with locally convex
boundary whose sectional curvatures are $\ge\kappa$. (Locally)
convex subsets of such Riemannian manifolds are also locally
$\textup{Alex}(\kappa)$. We shall be interested mostly in the case
when $\kappa=1$ and $M_1^n=\mathbb{S}^n$ -- locally convex subsets
of $\S^n$ are locally $\textup{Alex}(1)$.

The following is a local-to-global theorem for
$\textup{Alex}(\kappa)$ spaces which is analogous to the
Cartan-Hadamard theorem for $\textup{CAT}(\kappa)$ spaces with
$\kappa\le0$ (see for example \cite{bh}). Unlike the Cartan-Hadamard
theorem, however, we do not place any topological restriction on the
space in this theorem:

\begin{theorem}[Globalization Theorem] \label{thm:locglo}
If a complete path metric space is locally $\textup{Alex}(\kappa)$,
then it is $\textup{Alex}(\kappa)$ and has diameter $\le D_\kappa$.
\end{theorem}

\noindent For its proof we refer to \cite{bgp}. As a corollary of
the globalization theorem, we have the following criterion for
locally convex subsets of $M_\kappa^n$ to be convex. Note that if
$\kappa>0$, geodesics in $M_\kappa^n$ have length at most
$D_\kappa$.

\begin{corollary} \label{cor:locglo}
Let $C$ be a locally convex connected subset of $M_\kappa^n$. If
$\kappa>0$, we assume in addition that $C$ is not a $1$-dimensional
manifold. If $C$ is complete and locally compact with respect to the
induced path metric, then $C$ is convex in $M_\kappa^n$.
\end{corollary}

\begin{proof} Because $C$ is locally convex in $M_\kappa^n$ (and is
not a $1$-dimensional manifold in case $\kappa>0$), $C$ is locally
$\textup{Alex}(\kappa)$. If $C$ is complete with respect to the
induced length metric, the globalization theorem tells us that $C$
is an $\textup{Alex}(\kappa)$ space of diameter $\le D_\kappa$. Let
$p$ and $q$ be two points of $C$. Because $C$ is connected, complete
and locally compact with respect to the induced path metric, $C$
satisfies the assumption of the Hopf-Rinow Theorem (see for example
\cite{bh}) and hence there is a geodesic $[p,q]_C$ \emph{in} $C$
joining $p$ and $q$. As $C$ is locally convex, however, this curve
$[p,q]_C$ has to be a local geodesic in $M_\kappa^n$. Since $C$ has
diameter $\le D_\kappa$, the length of $[p,q]_C$ is at most
$D_\kappa$. It follows from the simple-connectedness of $M_\kappa^n$
that $[p,q]_C$ is a (global) geodesic in $M_\kappa^n$.
\end{proof}

\section{Main objects} \label{sec:complex}

We define metric polyhedral complexes which are locally isometric to
$\mathbb{S}^n$. Our presentation follows that of
$M_\kappa$--polyhedral complexes in \cite{bh}, where $\kappa=1$ in
our case. We consider subcomplexes of such polyhedral complexes that
embed isometrically into $\mathbb{S}^n$ as topological balls, and
present a convexity criterion for them. We also study special
subcomplexes called stars and residues.

\subsection{Complexes}

\begin{definition}[$n$-complexes] \label{def:complex}
Given a family $\{P_i : i\in \mathcal{I}\}$ of $n$-polytopes in
$\S^n$, let $X$ be a connected $n$-manifold (possibly with non-empty
boundary $\partial X$) which is obtained by gluing together members
of $\{P_i\}$ along their respective facets by isometries. We denote
by $\sim$ the equivalence relation on the disjoint union
$\bigsqcup_{i\in \mathcal{I}}P_i$ induced by this gluing so that
$$
X={\bigsqcup_{i\in \mathcal{I}} P_i} / \sim.
$$
Let $\pi:\bigsqcup_{i\in \mathcal{I}} P_i\to X$ be the natural
projection and denote $\pi_i=\pi\vert_{P_i}$. We call the manifold
$X$ a \emph{spherical polytopal $n$-complex} (\emph{$n$-complex},
for short) provided that
\renewcommand{\labelenumi}{(\arabic{enumi})}
\begin{enumerate}
\item
the family $\{\pi_i(P_i)\,|\,i\in \mathcal{I}\}$ is locally finite;
\item
it is endowed with the quotient metric associated to the projection
$\pi$;
\item
its interior $X^\circ$ is locally isometric to $\S^n$;
\item
it is simply-connected.
\end{enumerate}
\end{definition}

For each $n$-complex $X$ the conditions (3) and (4) guarantee that
there is an associated \emph{developing map}
$$
dev:X\to\S^n
$$
which is a local isometry on the interior of $X$ and which extends
naturally to the boundary of $X$. The developing map is well-defined
up to post-composition with an isometry of $\S^n$.

\begin{convention}
Whenever we mention an $n$-complex $X$, we shall tacitly assume that
a developing map $dev:X\to\S^n$ for $X$ is already chosen. Given a
subset $K\subset X$, we shall denote by $K_\S$ the image $dev(K)$ of
$K$ under this developing map $dev$.
\end{convention}

Let $X$ be an $n$-complex. A subset $f\subset X$ is called an
\emph{$m$-cell} if it is the image $\pi_i(f_i)$ for some $m$-face
$f_i$ of $P_i$; the interior of $f$ is the image under $\pi_i$ of
the interior of $f_i$. The $0$-cells, $1$-cells, $(n-2)$-cells and
$(n-1)$-cells of $X$ are also called \emph{vertices}, \emph{edges},
\emph{ridges} and \emph{facets} of $X$, respectively. Two $m$-cells
$f_1$ and $f_2$ of $X$ are said to be \emph{adjacent} if their
intersection $f_1\cap f_2$ is an $(m-1)$-cell of $X$. A
\emph{subcomplex} of $X$ is a union of cells of $X$.

\subsection{Links in complexes} \label{ssec:links}

Let $X$ be an $n$-complex. For each $m$-cell $e$ of $X$ with $m<n$,
we denote $\mathcal{I}(e)=\{i\in \mathcal{I}\,|\,e\subset
\pi_i(P_i)\}$. The link $\L(e;X)$ of $e$ in $X$ is an
$(n-m-1)$-complex defined as follows.

Let $\s$ be a facet of $X$ containing $e$ and let
$\mathcal{I}(\s)=\{j,k\}\subset \mathcal{I}(e)$. For each $i\in
\mathcal{I}(\s)$ let $e_i$ and $\s_i$ be faces of $P_i$ such that
$\pi_i(e_i)=e$ and $\pi_i(\s_i)=\s$. By definition of $n$-complex,
the facets $\s_j$ and $\s_k$ are isometric by an isometry
$\phi_{jk}$ which restricts to an isometry between $e_j$ and $e_k$.
Then $\phi_{jk}$ induces an isometry between $(\s_j)_{(e_j;P_j)}$
and $(\s_k)_{(e_k;P_k)}$. Because $(\s_i)_{(e_i;P_i)}$ is a facet of
the polytope $\L(e_i;P_i)$ for each $i\in \mathcal{I}(\s)$, this
shows that the equivalence relation $\sim$ on $\bigsqcup_{i\in
\mathcal{I}} P_i$ induces an equivalence relation $\sim_\s$ on
$\L(e_j;P_j)\bigsqcup\L(e_k;P_k)$. Combining all equivalence
relations $\sim_\s$ for all facets $\s$ of $X$ containing $e$, we
obtain an equivalence relation $\sim_e$ on $\bigsqcup_{i\in
\mathcal{I}(e)}\L(e_i;P_i)$. The \emph{link} $\L(e;X)$ of $e$ in $X$
is then defined as
$$
\L(e;X)=\bigsqcup_{i\in \mathcal{I}(e)}\L(e_i;P_i)/\sim_e
$$
and is an $(n-m-1)$-complex endowed with the quotient metric
associated to the natural projection $\bigsqcup_{i\in
\mathcal{I}(e)}\L(e_i;P_i)\to\L(e;X)$ induced by $\sim_e$. Indeed,
because $X$ is a manifold, if $e$ is contained in the boundary of
$X$ then the link $\L(e;X)$ is isometric to a ball in $\S^{n-m-1}$;
otherwise, it is isometric to the sphere $\S^{n-m-1}$. Thus it is
simply-connected and its interior is locally isometric to the sphere
$\S^{n-m-1}$.

Let $X$ be an $n$-complex. We can extend the identity \eqref{links2}
(which is obtained from Lemma~\ref{links1}) to the current setting
as follows. Let $e\subsetneq f$ be cells of $X$. Keeping the same
notation as above, we recall that the link $\L(e;X)$ is the quotient
of $\bigsqcup_{i\in \mathcal{I}(e)}\L(e_i;P_i)$ by $\sim_e$, where
$e_i$ is a face of $P_i$ such that $\pi_i(e_i)=e$ for each $i\in
\mathcal{I}(e)$. Consider $\mathcal{I}(f)=\{i\in
\mathcal{I}\,|\,f\subset \pi_i(P_i)\}\subset \mathcal{I}(e)$. For
each $i\in \mathcal{I}(f)$ let $f_i$ be the face of $P_i$ such that
$\pi_i(f_i)=f$. Now by Lemma~\ref{links1} we have that
$(f_i)_{(e_i;P_i)}$ is a face of $\L(e_i;P_i)$ for each $i\in
\mathcal{I}(f)$. Since $\sim_e$ identifies all $(f_i)_{(e_i;P_i)}$
for $i\in \mathcal{I}(f)$, we may define
\begin{align} \label{link-cell}
f_{(e;X)}=\pi_i({f_i}_{(e_i;P_i)})
\end{align}
for any chosen $i\in \mathcal{I}(f)$ and it follows that $f_{(e;X)}$
is a cell of the complex $\L(e;X)$. From the identity \eqref{links2}
we see that the equivalence relation $\sim_{f_{(e;X)}}$ on
$\bigsqcup_{i\in \mathcal{I}(f)}\L({f_i}_{(e_i;P_i)};\L(e_i;P_i))$,
which is by definition induced from $\sim_e$, is equal to the
equivalence relation $\sim_f$ on $\bigsqcup_{i\in
\mathcal{I}(f)}\L(f_i;P_i)$. It now follows that
\begin{align} \label{isom'}
\L(f;X)&=\bigsqcup_{i\in \mathcal{I}(f)}\L(f_i;P_i)/\sim_f \notag\\
&=\bigsqcup_{i\in\mathcal{I}(f)}\L({f_i}_{(e_i;P_i)};\L(e_i;P_i))/\sim_{f_{(e;X)}}\notag\\
&=\L(f_{(e;X)};\L(e;X)).
\end{align}

\subsection{Polyballs}

Recall that an $n$-complex is equipped with a developing map into
$\S^n$.
\begin{definition}[Polyballs] \label{def:polyball}
An \emph{$n$-polyball} $B$ is an $n$-complex which is topologically
an $n$-dimensional ball with boundary and whose developing map
$$
dev:B\hookrightarrow\S^n
$$
is an isometric embedding into $\S^n$. An $n$-polyball $B$ is said
to be \emph{convex} (resp. \emph{locally convex}) if its developing
image $B_\S=dev(B)$ is a convex (resp. locally convex) subset of
$\S^n$.
\end{definition}

Being compact, an $n$-polyball consists of a finite number of
$n$-cells. In particular, a single $n$-cell is itself an
$n$-polyball. If $X$ is an $n$-complex with boundary and $f$ is an
$m$-cell in the boundary of $X$, then the link $\L(f;X)$ is an
$(n-m-1)$-polyball.

Let $B$ be a fixed $n$-polyball from now on. Because $B$ consists of
a finite number of $n$-cells $P$ and because their images $P_\S$ are
compact convex subsets of $\S^n$, its image $B_\S$ in $\S^n$ is
compact with respect to the path metric induced from that of the
sphere $\S^n$. Thus if we know that $B$ is locally convex, then it
follows from Corollary~\ref{cor:locglo} (applied to $M_1^n=\S^n$)
that $B$ is convex. See Lemma~\ref{lem:ridge} below. Therefore, to
establish convexity of $B$, it suffices to investigate local
convexity of $B$.

Because the $n$-polyball $B$ is a manifold, its local convexity
matters only at its boundary points. Because of the polyhedral
structure of $B$, however, it suffices to investigate the links of
cells in the boundary of $B$. More precisely, let $x$ be a point in
the boundary of $B$. There is a unique cell $f$ of $B$ that contains
$x$ as its interior point. The local geometry of $B$ at $x$ is
completely determined by the union of $n$-cells containing $f$,
whose geometry is then captured by the link of $f$ in $B$. Thus
$B_\S$ is locally convex at $x_\S$ if and only if the link $\L(f;B)$
is a convex polyball. Therefore, $B$ is locally convex if and only
if the links $\L(f;B)$ are convex polyballs for all cells $f$ in the
boundary of $B$. This last condition holds for facets $\s$ in the
boundary of $B$ since the link $\L(\s;B)$ is just a singleton of
$\S^0$ and hence convex. Thus we are left with cells of dimension at
most $n-2$. It turns out that only $(n-2)$-cells, i.e. the ridges of
$B$, need to be investigated.

Let $f$ be an $m$-cell in the boundary of $B$. The link $\L(f;B)$ of
$f$ is an $(n-m-1)$-polyball. On the other hand, if $v$ is a vertex
of $f$, then $f$ descends to an $(m-1)$-cell $f_{(v;B)}$ in the link
$\L(v;B)$ of $v$. The link $\L(v;B)$ is an $(n-1)$-polyball with
$f_{(v;B)}$ in its boundary. From \eqref{isom'} of the previous
subsection, we have the following identity between the two
$(n-m-1)$-polyballs
\begin{align}\label{isom''}
\L(f;B)=\L(f_{(v;B)};\L(v;B)).
\end{align}
Therefore, the link $\L(v;B)$ of the vertex $v$ contains all the
information about the links $\L(f;B)$ of those cells $f$ which
contain $v$. In particular, if the link $\L(v;B)$ of $v$ is a convex
$(n-1)$-polyball then the link $\L(f;B)$ of $f$ is also a convex
$(n-m-1)$-polyball.

Conversely, the proof of the lemma below shows that if the links
$\L(e;B)$ are convex for all ridges $e$ of $B$ in the boundary of
$B$, then $\L(v;B)$ is convex for every boundary vertex $v$.

\begin{lemma} \label{lem:ridge}
Let $B$ be an $n$-polyball. If the links $\L(e;B)$ are convex for
all ridges $e$ contained in the boundary of $B$, then $B$ is convex.
\end{lemma}

\begin{proof}
We shall prove the lemma by induction on the dimension $n$ of $B$.
In the base case when $n=2$, the ridges of $B$ are just vertices of
$B$. From the above discussion we see that $B$ is locally convex. By
Corollary~\ref{cor:locglo}, $B$ is convex.

Suppose now that the assertion is true for polyballs of dimension
$\le n-1$. Let $B$ be an $n$-polyball and assume that the links
$\L(e;B)$ are convex for all ridges $e$ contained in the boundary of
$B$. Let $v$ be a vertex in the boundary of $B$. Then the link
$\L(v;B)$ is an $(n-1)$-polyball and its ridges are those
$e_{(v;B)}$ which come from the ridges $e$ of $B$ that contain $v$.
The ridges $e_{(v;B)}$ are in the boundary of $\L(v;B)$ if and only
if the ridges $e$ are in the boundary of $B$. Because $\L(e;B)$ is
assumed to be convex, it follows from \eqref{isom''} that
$\L(e_{(v;B)};\L(v;B))$ is convex, too. Hence the induction
hypothesis applies and we conclude that $\L(v;B)$ is convex. Since
$v$ is arbitrary, this implies that $B$ is locally convex. By
Corollary~\ref{cor:locglo} once again, we conclude that $B$ is
convex. The induction steps are complete.
\end{proof}

\subsection{Stars and residues}

Let $X={\bigsqcup_{i\in I} P_i} / \sim$ be a fixed $n$-complex
throughout this subsection. We shall define two kinds of
subcomplexes of $X$ called stars and residues. In most cases later
on they will be $n$-polyballs in their own right.

\begin{definition}[Stars and residues\footnote{Our definition of
star seems to be somewhat non-standard. We borrowed the term
"residue" from \cite{js2}, where residues are defined in the same
way as in the present paper.}] \label{def:star-residue} Let
$Y\subset X$ be a subcomplex and let $\s\subset Y$ be a cell or a
subcomplex of $X$.
\renewcommand{\labelenumi}{\textup{(\arabic{enumi})}}
\begin{enumerate}
\item
The \emph{star} $st(\s;Y)$ of $\s$ in $Y$ is the union of the cells
of $Y$ that \textsl{intersect} $\s$.
\item
The \emph{residue} $res(\s;Y)$ of $\s$ in $Y$ is the union of the
cells of $Y$ which \textsl{contain} $\s$.
\end{enumerate}
We set $st^0(\s;Y)=\s$ and define $st^{k+1}(\s;Y)=st(st^k(\s;Y);Y)$
inductively. In case $Y=X$ we simply denote $st^k(\s)=st^k(\s;X)$
and $res(\s)=res(\s;X)$. Notice that $st(v)=res(v)$ for vertices $v$
of $X$.
\end{definition}

Let $Y_1$ and $Y_2$ be subcomplexes of $X$. The following relations
are immediate from the definition of star.
\begin{align}
st(Y_1\cup Y_2)=st(Y_1)\cup st(Y_2);\label{st1}\\
st(Y_1\cap Y_2)\subset st(Y_1)\cap st(Y_2).\label{st2}
\end{align}

Iterated stars satisfy the following properties. Let $P_0$ be an
$n$-cell in $X$ and let $\mathcal{V}$ be the set of all vertices in
$P_0$. It follows directly from the definition that
\begin{align} \label{*}
P_0=\bigcap_{v\in\mathcal{V}}st(v)\quad\text{and}\quad
st(P_0)=\bigcup_{v\in\mathcal{V}}st(v).
\end{align}
Let $\mathcal{P}$ be the set of all $n$-cells in $st(P_0)$. We claim
that for each $k\ge1$
\begin{align} \label{**}
P_0\subset\bigcap_{P\in\mathcal{P}}st^k(P)\quad\text{and}\quad
st^{k+1}(P_0)=\bigcup_{P\in\mathcal{P}}st^k(P).
\end{align}
The former inclusion is obvious. We can see the latter equality
using induction on $k$. The base case $k=1$ follows immediately from
the definition. Suppose it is true up to $k-1$. We then have
$st^{k+1}(P_0)=st(st^k(P_0))=st(\bigcup_\mathcal{P}
st^{k-1}(P))=\bigcup_\mathcal{P} st(st^{k-1}(P))=\bigcup_\mathcal{P}
st^k(P)$, where the third equality follows from \eqref{st1}. See
Figure~\ref{fig:stars} (a). Using properties \eqref{*} and
\eqref{**} we can prove the following lemma.

\begin{figure}[htbp]
\includegraphics[width=1\linewidth]{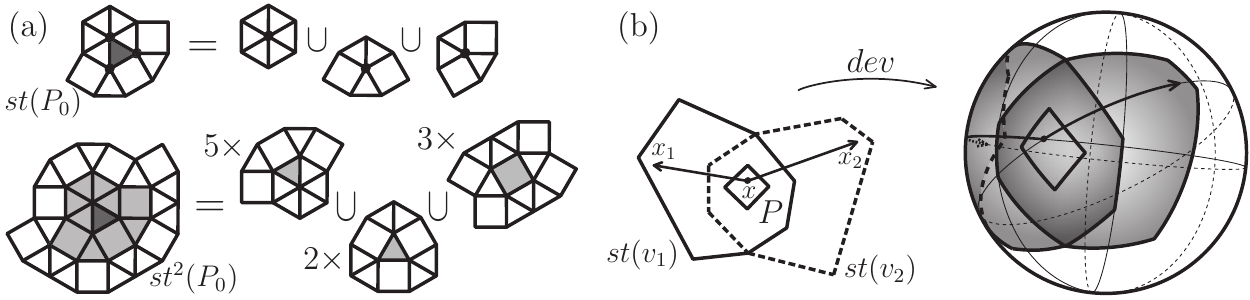}
\caption{\footnotesize (a) Illustrations of \eqref{*} and
\eqref{**}. (b) Proof of Lemma~\ref{lem:stars}.} \label{fig:stars}
\end{figure}

\begin{lemma} \label{lem:stars}
Let $X$ be an $n$-complex.
\renewcommand{\labelenumi}{\textup{(\arabic{enumi})}}
\begin{enumerate}
\item
If $st(v)$ is a convex $n$-polyball for all vertices $v$ of $X$,
then $st(P)$ is an $n$-polyball for each $n$-cell $P$ in $X$.
\item
For each fixed $k\ge1$, if $st^k(P)$ is a convex $n$-polyball for
all $n$-cells $P$ in $X$, then $st^{k+1}(P)$ is an $n$-polyball.
\end{enumerate}
\end{lemma}

\begin{proof}
Recall that we have a developing map $dev:X\to\S^n$ of the
$n$-complex $X$ and we denote $K_\S=dev(K)$ for $K\subset X$.

(1) Let $P$ be an $n$-cell of $X$. Let $x_1,x_2\in st(P)$ be such
that $x_1\neq x_2$. We want to show that $(x_1)_\S\neq (x_2)_\S$.
Let $\mathcal{V}$ be the set of all vertices in $P$. The second
identity of \eqref{*} implies that there are vertices
$v_1,v_2\in\mathcal{V}$ such that $x_1\in st(v_1)$ and $x_2\in
st(v_2)$. If $x_1,x_2\in st(v_1)\cap st(v_2)$ then $(x_1)_\S\neq
(x_2)_\S$, because $st(v_1)\cap st(v_2)\subset st(v_1)$ and
$st(v_1)$ is a polyball and hence $dev|_{st(v_1)}$ is an embedding.
Thus we may assume from now on that $x_1\in st(v_1)\setminus
st(v_2)$ and $x_2\in st(v_2)\setminus st(v_1)$. See
Figure~\ref{fig:stars} (b).

Fix $i=1,2$. Consider the interior $P^\circ$ of $P$ and choose a
point $x\in P^\circ$. Consider the geodesic segment
$[(x)_\S,(x_i)_\S]$ in $\S^n$. Because $st(v_i)$ is a convex
polyball and because $(x)_\S\in(P^\circ)_\S\subset st(v_i)_\S$ by
the first identity of \eqref{*}, we must have that
$$
[(x)_\S,(x_i)_\S]\subset st(v_i)_\S.
$$
Furthermore, the length of $[(x)_\S,(x_i)_\S]$ is less than $\pi$,
since the diameter of the convex (proper) subset $st(v_i)_\S$ is at
most $\pi$ and $(x)_\S$ is an interior point of $st(v_i)_\S$.

If the initial directions at $(x)_\S$ of $[(x)_\S,(x_1)_\S]$ and
$[(x)_\S,(x_2)_\S]$ coincide, say,
$$
[(x)_\S,(x_1)_\S]\subset[(x)_\S,(x_2)_\S]\subset st(v_2)_\S,
$$
then we have $(x_1)_\S\in st(v_2)_\S$, contradictory to $x_1\in
st(v_1)\setminus st(v_2)$. Thus the initial directions at $(x)_\S$
of the two geodesic segments must be different. Because their
lengths are less than $\pi$, however, this implies that they
intersect only at $(x)_\S$, hence $(x_1)_\S\neq (x_2)_\S$.

Thus we have shown that $dev$ is injective when restricted to the
star $st(P)$. The identities in \eqref{*} again imply that
$st(P)_\S$ is a union of convex subsets $st(v)_\S$ whose
intersection has non-empty interior $(P^\circ)_\S$. Therefore, the
image $st(P)_\S$ is a topological ball, and this completes the proof
that $st(P)$ is an $n$-polyball.

(2) For each fixed $k\ge1$, the proof goes word-by-word in the same
manner as in (1), except we need to use \eqref{**} instead.
\end{proof}

The residue of a cell $e$ serves as a nice neighborhood of the
interior points of $e$. For example, let $B\subset X$ be a
subcomplex which is an $n$-polyball. If $e$ is a cell in the
boundary of $B$ and $x$ is an interior point of $e$, then $res(e;B)$
is a neighborhood of $x$ in $B$. Because the link of $e$ in $B$
depends only on the union of cells in $B$ that contain $e$, we have
$\L(e;B)=\L(e;res(e;B))$. Therefore, once we know that $res(e;B)$ is
a convex polyball, then we can conclude that $\L(e;B)$ is convex.

In view of Lemma~\ref{lem:ridge}, however, it is important for us to
study the residues of ridges of $X$. So let $e$ be a ridge of $X$
and consider its residue $res(e)=res(e;X)$. Because ridges are
$(n-2)$-dimensional, the link $\L(e,X)$ of $e$ is a $1$-complex
embedded in $\S^1$ with its vertices and $1$-cells coming from
$(n-1)$-cells and $n$-cells of $X$ containing $e$, respectively (see
\eqref{link-cell}). Indeed, the link $\L(e;X)$ is a circular arc or
the whole $\S^1$ depending on whether $e$ is in the boundary of $X$
or not. Thus we can give a linear (or cyclic) order in the set of
$n$-cells in $res(e)$ so that
\begin{align} \label{cyclic}
res(e)=P_1\cup P_2\cup\cdots\cup P_{d_e},
\end{align}
where $P_i$ and $P_{i+1}$ are adjacent and share a common facet
$\s_i=P_i\cap P_{i+1}$ (the indices are taken modulo $d_e$ in case
$\L(e;X)=\S^1$) and $\s_i\cap\s_j=e$ for $i\neq j$.

We conclude this section with the following property of residues,
which will lead to the definition of residual convexity in the next
section. Let $0\le m\le n-1$. Let $f$ be an $(m+1)$-cell of $X$ and
$\mathcal{H}$ be the set of all $m$-cells $h$ in $f$. We then have
\begin{align} \label{***}
res(f)=\bigcap_{h\in\mathcal{H}} res(h).
\end{align}
Indeed, the inclusion $res(f)\subset\bigcap_{h\in\mathcal{H}}
res(h)$ is clear. If $\s\subset\bigcap_{h\in\mathcal{H}} res(h)$ is
a cell, then $\s$ contains all $m$-cells in $f$. Thus $\s$
necessarily contains $f$ and hence $\s\subset res(f)$.

\section{Convexity} \label{sec:convexity}

This is the main section of the paper. Here we consider only those
$n$-complexes $X$ which have empty boundary. We shall introduce
local convexity conditions on $X$ called residual convexity and
strong residual convexity. Combined with the global condition that
$X$ is without boundary, these conditions enable us to show that $X$
is isometric to a convex proper domain in $\S^n$. We also provide a
simple combinatorial condition for a residually convex complex to be
strongly residually convex.

\subsection{Main theorem}

\begin{lemma} \label{lem:residue}
Let $X$ be an $n$-complex without boundary. The following conditions
on $X$ are equivalent to each other.
\renewcommand{\labelenumi}{\textup{(\arabic{enumi})}}
\begin{enumerate}
\item
The star $st(v)=res(v)$ is a convex $n$-polyball for every vertex
$v$ of $X$.
\item
For each fixed $k$ with $1\le k\le n-2$, the residue $res(f)$ is a
convex $n$-polyball for every $k$-cell $f$ of $X$.
\item
The residue $res(\s)$ is a convex $n$-polyball for every facet $\s$
of $X$.
\end{enumerate}
\end{lemma}

\begin{proof}
Because the intersection of convex subsets is again a convex subset,
the implications (1)$\Rightarrow$(2) and (2)$\Rightarrow$(3) follow
from \eqref{***} inductively. In fact, these implications are true
without the assumption that $X$ is without boundary, which is needed
only in the proof of (3)$\Rightarrow$(1).

We first observe the following fact for an $n$-complex $X$ with or
without boundary. Namely, we claim that for each vertex $v$ of $X$
the star $st(v)$ is an $n$-polyball. The proof is essentially the
same as the proof of Lemma~\ref{lem:stars}. Let $dev:X\to\S^n$ be a
developing map of $X$ and let $x_1,x_2\in st(v)$ be such that
$x_1\neq x_2$. Fix $i=1,2$. There is an $n$-cell $P_i$ of $X$ such
that $x_i\in P_i$ and $v$ is a vertex of $P_i$. Because $n$-cells
are polyballs, to show injectivity of $dev$ we may assume that
$x_1\in P_1\setminus P_2$ and $x_2\in P_2\setminus P_1$. Now,
consider the geodesic segment $[(v)_\S,(x_i)_\S]$ in $\S^n$. Because
$n$-cells are convex polyballs, we must have that
$[(v)_\S,(x_i)_\S]\subset (P_i)_\S$. Furthermore, the length of
$[(v)_\S,(x_i)_\S]$ is less than $\pi$, since an $n$-cell is
contained in an open halfspace of $\S^n$. As in the proof of
Lemma~\ref{lem:stars}, the initial directions at $(v)_\S$ of the two
geodesic segments must be different. Because their lengths are less
than $\pi$, however, this implies that they intersect only at
$(v)_\S$, hence $(x_1)_\S\neq (x_2)_\S$. Thus $dev$ is injective
when restricted to $st(v)$. Furthermore, because $X$ is a manifold,
the image $st(v)_\S$ has to be a topological ball. This completes
the proof of the claim. Notice that the vertex $v$ is an interior
(resp. boundary) point of the $n$-polyball $st(v)$, if it is an
interior (resp. boundary) point of $X$.

We now begin the proof of (3)$\Rightarrow$(1). Assume the condition
(3). Because $X$ is without boundary, each vertex $v$ of $X$ is an
interior point of the $n$-polyball $st(v)$. Let $e$ be a ridge of
$X$ in the boundary of $st(v)$. Then $e$ does not contain $v$. We
claim that the $res(e;st(v))$ is either a single $n$-cell or a union
of two adjacent $n$-cells. Indeed, if there is no facet of $X$
containing both $v$ and $e$, then $e$ intersects only a single
$n$-cell in $st(v)$, which is $res(e;st(v))$. If $\s$ is a facet of
$X$ containing both $v$ and $e$, then $e$ intersects two adjacent
$n$-cells in $st(v)$, whose union is $res(e;st(v))=res(\s)$. This
proves the claim. In both cases, the condition (3) implies that the
$res(e;st(v))$ is a convex $n$-polyball. Therefore, the link
$\L(e;st(v))$ is convex. Since $e$ is arbitrary, it follows from
Lemma~\ref{lem:ridge} that the $n$-polyball $st(v)$ is convex.
\end{proof}

\begin{definition}[Residual convexity] \label{def:rconvexity}
An $n$-complex $X$ is said to be \emph{residually convex} if it is
without boundary and if it satisfies one of the equivalent
conditions in the previous lemma.
\end{definition}

\begin{remark}
The condition (3) in Lemma~\ref{lem:residue} is the one that we
considered in the introduction. Kapovich introduced this condition
in \cite{kapovich}. The condition (3) is seemingly the weakest among
those listed in Lemma~\ref{lem:residue}, hence the easiest to
verify. Thus we shall verify the condition (3) whenever we want to
show residual convexity of a given $n$-complex.
\end{remark}

If $X$ is residually convex and $e$ is a ridge of $X$, then the
residue $res(e)$ is a (convex) $n$-polyball by
Lemma~\ref{lem:residue} (2). A subset $F$ of the boundary of
$res(e)$ is said to be convex if $F_\S$ is a convex subset of
$\S^n$.

\begin{definition}[Good ridges] \label{def:good}
A ridge $e$ of a residually convex $n$-complex $X$ is said to be
\emph{good} if its residue $res(e)$ in $X$ has the following
property:
\begin{quote}
for every convex subcomplex $F$ in the boundary of $res(e)$ that
does not intersect $e$, the intersection $st(F)\cap res(e)$ is a
convex $n$-polyball.
\end{quote}
A ridge is \emph{bad} if it is not good.
\end{definition}

\begin{example}
See Figure~\ref{fig:strong} in the introduction. In this figure, a
ridge $e$ and its residue $res(e)$ are specified. The residue
$res(e)$ has five maximal convex subcomplexes $F$ in its boundary,
for each of which the intersection $st(F)\cap res(e)$ is shaded. The
picture marked with (*) shows that the intersection $st(F)\cap
res(e)$ is not convex for some $F$. Therefore, the ridge $e$ is bad.
Some more examples of good and bad ridges can be seen in
Figure~\ref{fig:wallpaper} below.
\end{example}

\begin{definition}[Strong residual convexity] \label{def:strong}
An $n$-complex $X$ is said to be \emph{strongly residually convex}
if it is residually convex and all ridges of $X$ are good.
\end{definition}

We shall discuss this property later after the main theorem (see
Remark~\ref{rem:strong}). The proof of the following lemma is the
only place where strong residual convexity is used explicitly, and
is illustrated by Figure~\ref{fig:strong} (with $st^k(P_0)$ playing
the role of $B$).

\begin{lemma} \label{lem:main}
Let $X$ be a strongly residually convex $n$-complex. Let $B$ be a
subcomplex of $X$ which is a convex $n$-polyball. If the star
$st(B)$ is an $n$-polyball then it is a convex $n$-polyball.
\end{lemma}

\begin{proof}
Let $e$ be a ridge in the boundary of $st(B)$. In view of
Lemma~\ref{lem:ridge} it suffices to show that the link
$\L(e;st(B))$ is convex, because the star $st(B)$ is assumed to be
an $n$-polyball. To see this, consider the residue $res(e)$ of $e$
in $X$, which is a convex $n$-polyball by residual convexity of $X$.
The subcomplex $B$ is also a convex $n$-polyball by assumption.
Because $e$ does not intersect $B$, the two $n$-polyballs $res(e)$
and $B$ intersect along their boundaries. Therefore, the
intersection $res(e)\cap B$ is a convex subcomplex in the boundary
of $res(e)$ that does not intersect $e$. From the strong residual
convexity of $X$ it follows that $st[res(e)\cap B]\cap res(e)$ is a
convex $n$-polyball.

We now claim that
$$
st[res(e)\cap B]\cap res(e)=res(e)\cap st(B).
$$
First, we have that
$$
st[res(e)\cap B]\cap res(e)\subset st[res(e)]\cap st(B)\cap
res(e)=res(e)\cap st(B),
$$
where the inclusion follows from \eqref{st2}. To show the reverse
inclusion, let $f$ be a cell in $res(e)\cap st(B)$. Then $f$ is in
$res(e)$ and intersects $B$. Thus $f\cap[res(e)\cap B]=f\cap B$ is
non-empty, and hence $f\subset st[res(e)\cap B]$. This proves the
claim.

As a result of the claim, we have that $res(e;st(B))=res(e)\cap
st(B)$ is a convex $n$-polyball. Therefore, the link $\L(e;st(B))$
is convex as desired.
\end{proof}

We are now ready to prove the main theorem of this paper.

\begin{theorem} \label{thm:main}
Let $X$ be an $n$-complex. If $X$ is strongly residually convex,
then $X$ is isometric to a convex proper domain in $\S^n$. In
particular, $X$ is contractible.
\end{theorem}

\begin{proof} By Lemma~\ref{lem:residue} (1), the star $st(v)$ is a convex
$n$-polyball for all vertices $v$ in $X$. Lemma~\ref{lem:stars} (1)
then says that the star $st(P)$ is an $n$-polyball for every
$n$-cell $P$ in $X$. By Lemma~\ref{lem:main}, it is a convex
$n$-polyball.

We next claim that $st^k(P)$ is a convex $n$-polyball for all
$k\ge1$ and for every $n$-cell $P$ in $X$. The proof goes by
induction on $k$. We just showed above that the base case $k=1$
holds true. Suppose that the claim is true for $k$, that is,
$st^k(P)$ is a convex $n$-polyball for every $n$-cell $P$ in $X$.
Then it follows from Lemma~\ref{lem:stars} (2) and
Lemma~\ref{lem:main} that $st^{k+1}(P)$ is a convex $n$-polyball for
each $n$-cell $P$ in $X$. The induction is complete.

Now it is easy to see that $dev:X\to\S^n$ is an embedding and $X_\S$
is a convex proper domain of $\S^n$. Consider the iterated stars
$st^k(P_0)$ of a fixed $n$-cell $P_0$ of $X$. Then for any two
distinct points $x_1\neq x_2$ of $X$, there is an integer $K\ge0$
such that $x_1,x_2\in st^K(P_0)$. Because $st^K(P_0)$ is a polyball,
we have $(x_1)_\S\neq(x_2)_\S$. Thus $dev:X\to\S^n$ is injective.
Moreover, because $st^K(P_0)$ is a convex polyball, the geodesic
segment $[(x_1)_\S,(x_2)_\S]$ is in $st^K(P_0)_\S\subset X_\S$.
Therefore, $X_\S$ is a convex subset of $\S^n$. Furthermore, because
all the images $st^k(P_0)$ are disjoint from the antipodal set
$-(P_0)_\S$, $X_\S$ is a proper subset of $\S^n$. Finally, because
$X$ is a connected $n$-manifold without boundary, the image $X_\S$
must be a connected open subset of $\S^n$. The proof is complete.
\end{proof}

\begin{remark} \label{rem:strong}
As its name suggests, strong residual convexity is indeed a very
strong local requirement for a few reasons;

(1) Essentially, we proved convexity of a subset $C\subset\S^n$ by
showing that $C$ is exhausted by a nested sequence of convex subsets
$U_k$ of $\S^n$. But, given a nested sequence of subsets $U_k$ which
exhausts $C$, the following weaker property would suffice to
guarantee convexity of $C$: for each $k$ there is $K>k$ such that
$$
\textup{conv}(U_k)\subset U_K.
$$
However, it seems hard to find \textsl{local} conditions which imply
this property.

(2) Moreover, a convex domain may admit residually convex
tessellations which are not strongly residually convex.
Figure~\ref{fig:wallpaper} shows examples of such tessellations of
the plane.
\begin{figure}[htbp]
\includegraphics[width=1\linewidth]{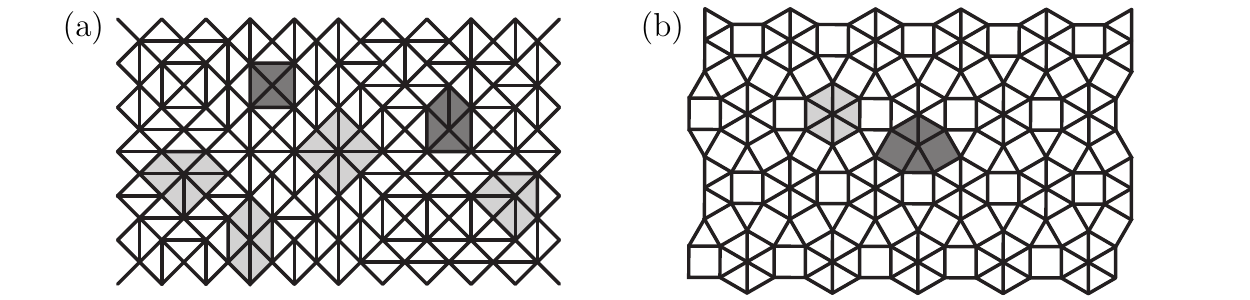}
\caption{\footnotesize Residually convex tessellations of the plane.
Residues of good ridges are shaded light. Residues of bad ridges are
shaded dark. (a) A random tessellation by right isosceles triangles.
Vertices of valency less than $6$ are bad ridges. (b) Vertices of
squares are bad ridges.} \label{fig:wallpaper}
\end{figure}
One may observe that triangles contribute to such phenomena; this is
the subject of the next subsection. Bounded convex domains may also
admit such tessellations. For example, consider the tessellations of
the Klein (projective) model of the hyperbolic plane corresponding
to the triangle reflection groups $G(a,b,c)$ where $a=2$. In such
tessellations, all $4$-valent vertices are bad ridges.
\end{remark}

Later, we shall need the following fact that residual convexity is
inherited by links.

\begin{lemma} \label{lem:link is residually convex}
Let $X$ be an $n$-complex and $e$ an $m$-cell of $X$ with $m<n$. If
$X$ is residually convex then the link $\L(e;X)$ is residually
convex.
\end{lemma}

\begin{proof}
If $X$ is residually convex then $X$ is without boundary and the
link $\L(e;X)$ is isometric to the sphere $\S^{n-m-1}$ (hence
without boundary). Note first that every cell of the link $\L(e;X)$
is of the form $f_{(e;X)}$ for some cell $f$ of $X$. See
\eqref{link-cell}. To check condition (3) in
Lemma~\ref{lem:residue}, let $\s_{(e;X)}$ be a facet of $\L(e;X)$
where $\s$ is a facet of $X$ containing $e$. Because an $n$-cell $P$
of $X$ contains $\s$ if and only if the corresponding $(n-m-1)$-cell
$P_{(e;X)}$ of $\L(e;X)$ contains $\s_{(e;X)}$, we see that the
residue of $\s_{(e;X)}$ in $\L(e;X)$ is equal to the link of $e$ in
$res(\s;X)$, that is,
$$
res(\s_{(e;X)};\L(e;X))=\L(e;res(\s;X)).
$$
Because $X$ is residually convex, however, the residue $res(\s;X)$
is a convex $n$-polyball and hence the link $\L(e;res(\s;X))$ is
also a convex $(n-m-1)$-polyball. The proof is complete.
\end{proof}

\subsection{Complexes without triangular polytopes}

We shall provide a simple combinatorial condition under which a
given residually convex $n$-complex $X$ becomes strongly residually
convex. In the following definition we regard a single polytope as a
complex and its boundary as a subcomplex.

\begin{definition}[Triangular polytopes] \label{def:triangular}
A polytope $P$ is said to be \emph{triangular} if it has a ridge $e$
and a face $f$ such that $res(e;\partial P)\cap f$ is disconnected.
Such a pair $(e,f)$ is called a \emph{triangularity pair} for $P$.
\end{definition}

Of course, triangles are the only triangular $2$-polytopes. More
discussion on (non-)triangular polytopes will be given after the
proof of the following theorem.

\begin{theorem} \label{thm:notriangle}
Let $X$ be a residually convex $n$-complex. If none of the $n$-cells
of $X$ is triangular, then $X$ is strongly residually convex.
\end{theorem}

\begin{proof}
Let $e$ be a ridge of $X$ and let $F$ be a convex subcomplex in the
boundary of $res(e)$ that does not intersect $e$. We shall show
below that $F$ intersects either a single $n$-cell in $res(e)$ or
two adjacent $n$-cells in $res(e)$ that share a common facet. It
then follows that $st(F)\cap res(e)$ is a single $n$-cell or the
residue of a facet. Because $X$ is residually convex,
Lemma~\ref{lem:residue} (3) implies that $st(F)\cap res(e)$ is a
convex $n$-polyball in either case, and we conclude that $e$ is a
good ridge. Since $e$ is arbitrary, it then follows that $X$ is
strongly residually convex.

As we observed in \eqref{cyclic}, we may set
$$
res(e)=P_1\cup P_2\cup\cdots\cup P_{d_e}
$$
so that $P_i$ and $P_{i+1}$ are adjacent and share a common facet
$\s_i=P_i\cap P_{i+1}$, where the indices are taken modulo $d_e$.
Moreover, we have $\s_i\cap\s_j=e$ for $i\neq j$. Because $F$ is a
convex subcomplex in the boundary of $res(e)$ and $F$ does not
intersect $e$, after cyclically permutating the indices of $P_i$, we
may further assume that $F$ decomposes into
$$
F=f_1\cup f_2\cup\cdots\cup f_d
$$
for some $d<d_e$, where we define $f_i=F\cap P_i\neq\emptyset$. See
Figure~\ref{fig:notriangle} (a).
\begin{figure}[htbp]
\includegraphics[width=1\linewidth]{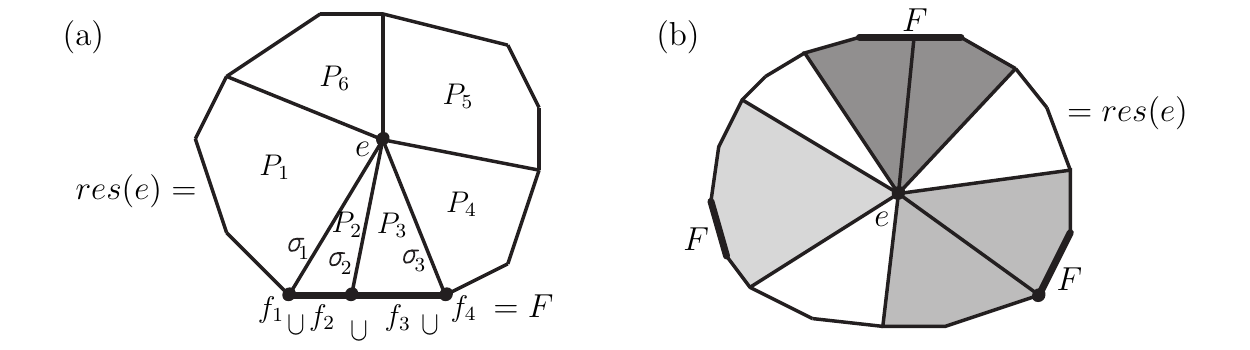}
\caption{\footnotesize (a) Proof of Theorem~\ref{thm:notriangle}. In
this picture the convex subcomplex $F$ in the boundary of $res(e)$
intersects $P_1,\ldots,P_4$, hence $d=4$. As the proof shows, the
polytope $P_2$ (as well as $P_3$) is triangular. (b) Illustration of
Theorem~\ref{thm:notriangle}. If there is no triangular polytope in
$res(e)$, then convex subcomplexes $F$ in its boundary intersect at
most two $n$-cells. It follows from residual convexity that the
ridge $e$ is good.} \label{fig:notriangle}
\end{figure}
We then observe the following:
\begin{itemize}
\item
For each $1\le i\le d$ the cell $f_i$ is convex because $F$ and
$P_i$ are convex. If the dimension of $f_i$ is $m$, then $f_i$ is a
single $m$-cell in $P_i$ because $P_i$ is a (convex) polytope;
\item
For each $1\le i\le d-1$ the intersection $f_i\cap f_{i+1}$ is a
non-empty subset of $\s_i$, because $F$ is connected and
$$
f_i\cap f_{i+1}=(F\cap P_i)\cap(F\cap P_{i+1})=F\cap(P_i\cap
P_{i+1})=F\cap\s_i\subset\s_i.
$$
\end{itemize}

Suppose now that $F$ intersects more than two $n$-cells in $res(e)$,
that is, $d\ge3$. We then have $f_1\cap f_2\subset\s_1$ and $f_2\cap
f_3\subset\s_2$. Because $\s_1\cap\s_2=e$ and $F$ does not intersect
$e$, we see that $f_1\cap f_2$ and $f_2\cap f_3$ are disjoint.
However, since $\s_1\cap f_2\subset P_1\cap F=f_1$ and $\s_2\cap
f_2\subset P_3\cap F=f_3$, we have $\s_1\cap f_2=f_1\cap f_2$ and
$\s_2\cap f_2=f_2\cap f_3$. It follows that
\begin{align*}
res(e,\partial P_2)\cap f_2&=(\s_1\cup\s_2)\cap f_2\\&=(\s_1\cap
f_2)\cup(\s_2\cap f_2)\\&=(f_1\cap f_2)\cup(f_2\cap f_3)
\end{align*}
is disconnected; a contradiction because $e\subset P_2$ is a ridge,
$f_2\subset P_2$ is a single $m$-cell, and $P_2$ is not triangular.
Therefore, we must have $d\le2$ and $F$ intersects either $P_1$ or
$P_1\cup P_2=res(\s_1)$. This completes the proof of the assertion
at the beginning.
\end{proof}

Combining the above with Theorem~\ref{thm:main} we have the
following immediate corollary:

\begin{corollary} \label{cor:notriangle}
Let $X$ be a residually convex $n$-complex. If none of the $n$-cells
of $X$ is triangular, then $X$ is isometric to a convex proper
domain in $\S^n$. In particular, $X$ is contractible.
\end{corollary}

\begin{remark}
In fact, the proof of Theorem~\ref{thm:main} shows that the
conclusion of Corollary~\ref{cor:notriangle} is still valid when $X$
is allowed to have a single triangular polytope. Namely, we can take
the single triangular polytope to be the initial polytope $P_0$ in
the proof of Theorem~\ref{thm:main}.
\end{remark}

The following corollary provides us with a necessary condition for
residual convexity:

\begin{corollary}
Let $X$ be a residually convex $n$-complex and $e$ an $m$-cell of
$X$ with $m\le n-3$. Then the link $\L(e;X)$ contains a triangular
$(n-m-1)$-polytope.
\end{corollary}

\begin{proof}
By Lemma~\ref{lem:link is residually convex} the link $\L(e;X)$ is a
residually convex $(n-m-1)$-complex which is isometric to the sphere
$\S^{n-m-1}$. If $\L(e;X)$ contained no triangular polytope, then it
would be contractible by the previous corollary. Because spheres are
not contractible, the link $\L(e;X)$ must contain a triangular
polytope.
\end{proof}

Thus, for example, one cannot obtain a residually convex $3$-complex
by gluing together copies of octahedra only.

\begin{remark}
(1) The previous corollary suggests that it would be good if one
could catalogue all residually convex tessellations of the sphere
$\S^n$.

(2) As we observed in the introduction, a residually convex complex
may fail to be strongly residually convex if it contains triangular
polytopes. See Figure~\ref{fig:benoist} (b). See also
Remark~\ref{rem:strong} (2) and Figure~\ref{fig:wallpaper}, where we
provided some examples of residually convex tessellations of the
plane which are not strongly residually convex.

(3) It would be of independent interest to know if every (convex or
non-convex) domain can admit a residually convex tessellation. Note
that Figure~\ref{fig:domain} is just a feasible picture of a
non-convex domain admitting a residually convex tessellation. In
addition to Figure~\ref{fig:benoist} (b), we provide in
Figure~\ref{fig:benoist-var} more examples of non-convex domains
admitting a residually convex tessellation.
\begin{figure}[htbp]
\includegraphics[width=1\linewidth]{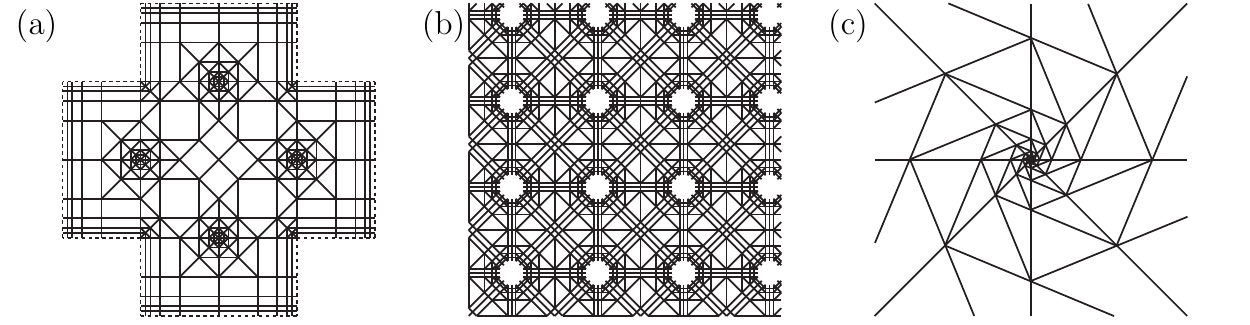}
\caption{\footnotesize Non-convex domains admitting a residually
convex tessellation. (a) A (cross-shaped) bounded domain with four
punctures. (b) The plane with a lattice of octagons removed. (c) A
variant of Benoist's example in Figure~\ref{fig:benoist} (b).}
\label{fig:benoist-var}
\end{figure}
\end{remark}

\begin{example}[Triangular polytopes] \label{exam:triangular}
(1) Triangles are the only triangular $2$-polytopes. Pyramids are
triangular; they are cone-like (see Definition~\ref{def:cone-like}
and Lemma~\ref{lem:cone-like}). Prisms over triangular polytopes are
also triangular because if $(e,f)$ is a triangularity pair for $P$
then so too is $(e\times I,f\times I)$ for $P\times I$.

(2) Let $P$ be an $n$-polytope and $v$ a vertex of $P$. If the link
$\L(v;P)$ is a triangular $(n-1)$-polytope then the polytope $P'$
obtained by truncating the vertex $v$ of $P$ is also triangular.
Indeed, if $e$ is a ridge and $f$ is a face of $P$ such that
$(e_{(v;P)},f_{(v;P)})$ is a triangularity pair for the link
$\L(v;P)$, then the pair of truncated faces $(e',f')$ is a
triangularity pair for $P'$. Thus, for example, if $v$ is a simple
vertex of $3$-polytope $P$, that is, $v$ is contained in exactly $3$
facets of $P$, then the polytope $P'$ obtained by truncating $v$ of
$P$ is triangular. (In this case, $(P')^*$ is also triangular.) See
Figure~\ref{fig:triangular} (a) and (b). Of course, not all
triangular polytopes are obtainable by this procedure. See
Figure~\ref{fig:triangular} (c).
\end{example}

\begin{figure}[htbp]
\includegraphics[width=1\linewidth]{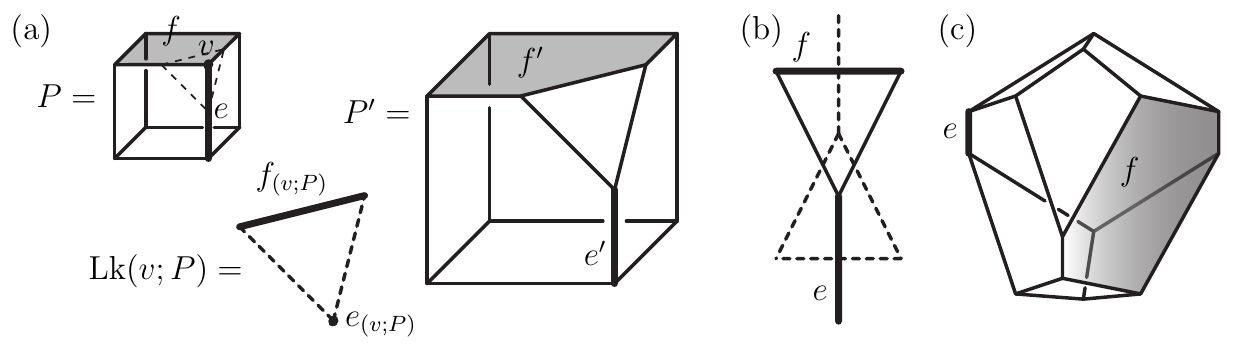}
\caption{\footnotesize (a) Illustration of the claim in
Example~\ref{exam:triangular} (2). The truncated cube $P'$ is
triangular. (b) If a $3$-polytope $P$ has a triangular facet with a
simple ($3$-valent) vertex, then its dual $P^*$ has the same
property. In this case, both $P$ and $P^*$ are triangular. (c) A
simple triangular $3$-polytope without triangular facets.}
\label{fig:triangular}
\end{figure}

\begin{example}[Non-triangular polytopes]
(1) Examples of non-triangular polytopes include $k$-gons ($k>3$),
Platonic solids other than tetrahedra, and prisms over
non-triangular polyhedra.

(2) One can transform any triangular polytope $P$ into a
non-triangular polytope as follows. Let $(e,f)$ be a triangularity
pair for $P$. The plan is to keep $e$ intact and break $f$ into
pieces so that no face of the new polytope $\hat{P}$ can give rise
to a triangularity pair with $e$. More precisely, let $f$ be a
minimal (with respect to inclusion) face of $P$ such that $(e,f)$ is
a triangularity pair for $P$. Place a vertex $v\in\S^n$ in the
exterior of $P$ arbitrarily close to the barycenter of $f$. The new
polytope $\hat{P}$ is obtained by "raising a pyramid" over the
residue $res(f;\partial P)$ with apex $v$. That is, we raise
pyramids with common apex $v$ over every face in the residue
$res(f;\partial P)$. See Figure~\ref{fig:raising}.
\begin{figure}[htbp]
\includegraphics[width=1\linewidth]{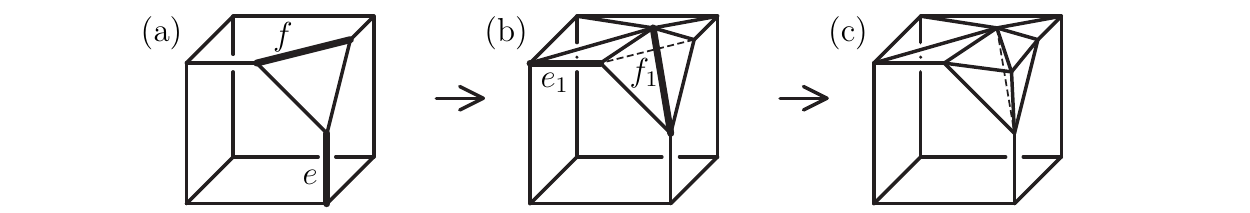}
\caption{\footnotesize (a) A truncated cube $P$ with triangularity
pair $(e,f)$. (b) Raising a pyramid over $res(f;\partial P)$, one
obtains a polytope $P'$ which has a triangularity pair $(e_1,f_1)$.
(c) Finally, raising a pyramid over $res(f_1;\partial P')$, one
obtains a polytope $P''$ which is non-triangular.}
\label{fig:raising}
\end{figure}
This procedure adds only a single vertex $v$ and does not change the
ridge $e$. If we keep doing this procedure for each minimal face $f$
with respect to $e$ and then the same procedure for all ridges $e$
of $P$, then we eventually get a non-triangular polytope.

(3) Similar reasoning shows that if we put new vertices $v_f$ over
all $i$-faces $f$ of $P$ ($i\neq0, n$) and raise pyramids
simultaneously over $f$ with apex $v_f$, then we get a
non-triangular polytope $\hat{P}$ whose boundary $\partial\hat{P}$
is combinatorially equal to the one which is obtained by performing
barycentric subdivision on the boundary $\partial P$ of the old
polytope $P$.

(4) Finally, in terms of duality the (non-)triangularity condition
translates as follows:
\begin{quote}
$P$ is non-triangular if and only if its dual $P^*$ satisfies the
property that, for each edge $e^*$ in $P^*$, the set
$st(e^*;\partial P^*)\setminus res(e^*;\partial P^*)$ is
disconnected.
\end{quote}
To see this, first notice that $e$ is a ridge of $P$ if and only if
$e^*$ is an edge of $P^*$. Indeed, $\s_1$ and $\s_2$ are facets of
$P$ such that $\s_1\cap\s_2=e$ if and only if $\s_1^*$ and $\s_2^*$
are vertices of $P^*$ spanning an edge $e^*$. In this case, we have
\begin{align*}
res(e^*;\partial P^*)&\subset st(\s_1^*;\partial P^*)\cap
st(\s_2^*;\partial P^*);\\
st(e^*;\partial P^*)&=st(\s_1^*;\partial P^*)\cup st(\s_2^*;\partial
P^*),
\end{align*}
which follows immediately from the definition. Because $P^*$ is a
(convex) polytope, however, the vertex stars $st(\s_1^*;\partial
P^*)$ and $st(\s_2^*;\partial P^*)$ are topological balls.
Therefore, the set $st(e^*;\partial P^*)\setminus res(e^*;\partial
P^*)$ is disconnected if and only if we have
$$
res(e^*;\partial P^*)=st(\s_1^*;\partial P^*)\cap st(\s_2^*;\partial
P^*).
$$

We now begin to prove the assertion made at the beginning. From the
previous discussion, we know that the set $st(e^*;\partial
P^*)\setminus res(e^*;\partial P^*)$ is connected for some edge
$e^*$ of $P^*$ if and only if there are faces $f_1^*$ and $f_2^*$ of
$P^*$ such that
\begin{itemize}
\item
for each $i=1,2$, $f_i^*\subset st(\s_i^*;\partial P^*)$, that is,
$\s_i^*$ is a vertex of $f_i^*$;
\item
$f^*:=f_1^*\cap f_2^*$ is not contained in $res(e^*;\partial P^*)$,
that is, there is no facet of $P^*$ containing both $e^*$ and $f^*$.
\end{itemize}
In terms of duality, this is equivalent to the condition that there
is a face $f$ of $P$ such that
\begin{itemize}
\item
for each $i=1,2$, $f_i$ is a face of the facet $\s_i$ of $P$;
\item
$f_1$ and $f_2$ are faces of $f$, and $f$ is disjoint from $e$.
\end{itemize}
In other words, there is a face $f$ of $P$ such that $f\cap\s_1=f_1$
is disjoint from $f\cap\s_2=f_2$, hence $(e,f)$ is a triangularity
pair for $P$ and $P$ is triangular.

(5) For example, let $P$ be a simple $n$-polytope, that is, every
vertex of $P$ is contained in exactly $n$ facets of $P$. Then the
facets of the dual $P^*$ are all $(n-1)$-simplices. Then the set
$st(e^*;\partial P^*)\setminus res(e^*;\partial P^*)$ is connected
for some edge $e^*$ if and only if either $P^*$ has a simple
$m$-simplex ($m<n-2$) or $\partial P^*$ has an edge-path of length
$3$ that does not bound a $2$-simplex. In conclusion, a simple
polytope $P$ is non-triangular if and only if $P$ has no $m$-simplex
($m>1$) and $\partial P^*$ has no nontrivial edge-path of length
$3$. Figure~\ref{fig:triangular} (c) shows a simple $3$-polytope
with no triangular facet but with a nontrivial edge-path of length
$3$ in the boundary of its dual.
\end{example}

\section{Proper convexity} \label{sec:pconvexity}

In this section we shall study only those residually convex
$n$-complexes $X$ which have no triangular $n$-cells. From
Corollary~\ref{cor:notriangle} we know that $X$ is isometric to a
convex proper domain in $\S^n$. Thus we may identify $X$ with its
image $dev(X)\subset\S^n$ and regard $X$ as a subset of $\S^n$. The
goal of this section is to prove the following theorem.

\begin{theorem} \label{thm:pconvexity}
Let $X\subset\S^n$ be a residually convex $n$-complex such that none
of the $n$-cells of $X$ are triangular. If $X$ has an $n$-cell $Q$
whose dual $Q^*$ is thick, then $X$ is a properly convex domain in
$\S^n$.
\end{theorem}

Before we proceed to prove the above theorem, we introduce thick
polytopes and discuss some of their examples.

\begin{definition}[Thick polytopes] \label{def:thick}
Let $P\subset\S^n$ be an $n$-polytope. We call $P$ \emph{thin}
provided that there is a hyperplane $H\subset\S^n$ (called a
\emph{cutting plane} for $P$) which contains no vertices of $P$ such
that the following condition is satisfied by all vertices $v$ of
$P$:
\begin{quote}
if the vertex $v$ is in one halfspace determined by $H$ then there
is another vertex $v'$ in the other halfspace that is connected to
$v$ by an edge.
\end{quote}
An $n$-polytope is said to be \emph{thick} if it is not thin.
\end{definition}

\begin{remark}
Of course, by dualizing Definition~\ref{def:thick}, we could state
Theorem~\ref{thm:pconvexity} without mentioning the dual $Q^*$ of
$Q$. We adopted the current approach, however, because the dualized
definition is less intuitive:
\begin{quote}
the dual $P^*$ of an $n$-polytope $P$ is thin if and only if there
is a point $x\in\S^n$ such that, for each facet $\s$ of $P$, the
hyperplane $\langle\s\rangle$ spanned by $\s$ does not contain $x$
and if $x$ is in the halfspace $\langle\s\rangle^\pm$ then $x$ is in
$\langle\s'\rangle^\mp$ for some facet $\s'$ adjacent to $\s$.
\end{quote}
\end{remark}

\begin{example}[Thin polytopes]
Figure~\ref{fig:thin} shows some examples of thin polytopes. It is
clear that triangles and quadrilaterals are the only thin
$2$-polytopes. Pyramids, bipyramids and prisms are thin (see
Lemma~\ref{lem:thin} below). The regular icosahedron is also thin.
\begin{figure}[htbp]
\includegraphics[width=1\linewidth]{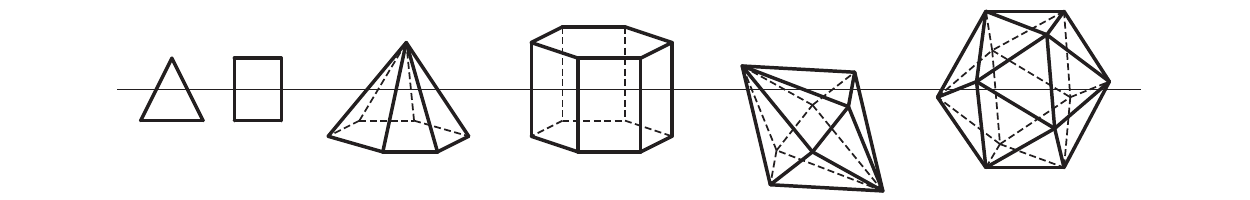}
\caption{\footnotesize Thin polytopes. The horizontal line
represents the cutting plane.} \label{fig:thin}
\end{figure}
\end{example}

\begin{remark}[Thick polytopes]
Definition~\ref{def:thick} suggests that polytopes with more
combinatorial complexity would have better chance to be thick and,
in some sense, thick polytopes are much more common than thin ones.
But it is rather hard to find simple combinatorial conditions which
imply thickness of polytopes.

In \cite{lee3} we classify thin simple $3$-polytopes and show that
they must contain a triangular or quadrilateral facet. Furthermore,
both thin simple $3$-polytopes and their dual polytopes turn out to
have Hamiltonian cycles. These facts imply that, for example,
dodecahedron, truncated icosahedron (soccer ball) and Tutte's
non-Hamiltonian simple polytopes are thick.
\end{remark}

To prove the above theorem we need some preparation. In the
following Sections~\ref{ssec:cone-like}-\ref{ssec:pav} we study more
about residually convex $n$-complexes without triangular $n$-cells
and develop a few related notions. The proof of
Theorem~\ref{thm:pconvexity} is then provided in the end of
Section~\ref{ssec:pav}.

\subsection{Cone-like polytopes} \label{ssec:cone-like}

The following definition and lemma are essential to the subsequent
constructions.

\begin{definition}[Cone-like polytopes] \label{def:cone-like}
A polytope $P$ is said to be \emph{cone-like} if it has a facet $\s$
such that $st(\s;\partial P)=\partial P$.
\end{definition}

Recall that the boundary $\partial P$ of a polytope $P$ is the union
of facets of $P$. Thus if $P$ is cone-like with respect to some
facet $\s$ then all facets of $P$ intersect $\s$. See
Figure~\ref{fig:cone-like}.

\begin{figure}[htbp]
\includegraphics[width=1\linewidth]{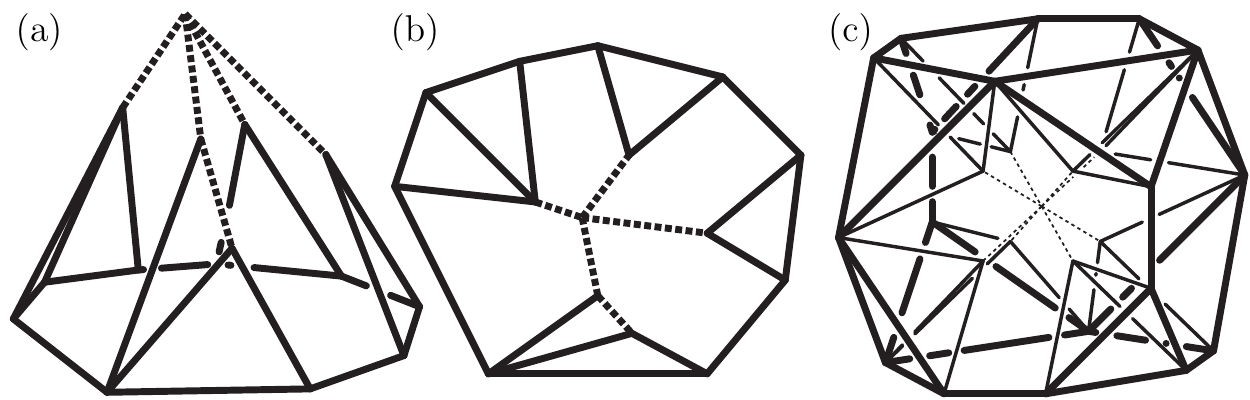}
\caption{\footnotesize Cone-like polytopes with $\Sigma$ dashed. (a)
A cone-like $3$-polytope. (b) A Schlegel diagram of the polytope in
(a). (c) A Schlegel diagram of a $4$-polytope whose $\Sigma$ is
$1$-dimensional. This $4$-polytope has $15$ facets.}
\label{fig:cone-like}
\end{figure}

\begin{lemma} \label{lem:cone-like}
Cone-like polytopes are triangular.
\end{lemma}

\begin{proof}
Suppose that $P$ is a cone-like $n$-polytope and $\s$ is a side of
$P$ such that $st(\s;\partial P)=\partial P$. The boundary $\partial
P$ of $P$ is topologically an $(n-1)$-dimensional sphere with cell
structure induced from the faces of $P$. Let $\Sigma\subset\partial
P$ be the union of all faces of $P$ that are disjoint from $\s$.
Because all facets of $P$ intersect $\s$, the dimension of $\Sigma$
is at most $n-2$.

\emph{Case I}. If $\Sigma$ has dimension $n-2$, choose any ridge $e$
of $P$ in $\Sigma$. Denote by $\s_1$ and $\s_2$ the two adjacent
facets of $P$ along $e$. Because $\s_1\cap\s_2=e$ is disjoint from
$\s$, we see that
\begin{align*}
res(e;\partial
P)\cap\s&=(\s_1\cup\s_2)\cap\s\\
&=(\s_1\cap\s)\cup(\s_2\cap\s)
\end{align*}
is disconnected. Therefore, $P$ is triangular.

\emph{Case II}. If $\Sigma$ has dimension $k$ with $k<n-2$, all
faces of $P$ of dimension $> k$ intersect $\s$. Let $\s'$ be a facet
of $P$ other than $\s$. Let $e=\s\cap \s'$ be a face of $P$. Because
all ridges of $P$ intersect $\s$, all facets of $\s'$ intersect $\s$
and hence $e$. Thus we have $st(e;\partial\s')=\partial\s'$. It
follows that $e$ is a facet of $\s'$ (hence, a ridge of $P$), $\s'$
is cone-like with respect to $e$ and that $res(e;\partial
P)=\s\cup\s'$. Now, because the dimension of $\Sigma$ is $k$, we can
choose a $(k+1)$-dimensional face $f$ of $P$ so that $\s'\cap
f\neq\emptyset$ and $\s'\cap f\subset\Sigma$. Then, because all
faces of $P$ of dimension $> k$ intersect $\s$, the intersection
$\s\cap f$ is non-empty and disjoint from $\s'\cap f\subset\Sigma$.
We thus have that
\begin{align*}
res(e;\partial
P)\cap f&=(\s\cup\s')\cap f\\
&=(\s\cap f)\cup(\s'\cap f)
\end{align*}
is disconnected and hence that $P$ is triangular.
\end{proof}

\begin{remark}
Not all triangular polytopes are cone-like. Such examples can be
seen in Figure~\ref{fig:triangular} (a) and (c).
\end{remark}

\subsection{Directed galleries and supporting hyperplanes}
\label{ssec:gallery}

From now on we assume that $X\subset\S^n$ is a residually convex
$n$-complex such that none of the $n$-cells of $X$ is triangular. It
follows from Lemma~\ref{lem:cone-like} that no $n$-cells of $X$ are
cone-like; this fact enables us to consider the following objects in
$X$.

We fix a specified $n$-cell $Q$ in $X$. Let $\s$ be a facet of $Q$.
Then there is an $n$-cell $P_1$ of $X$ adjacent to $Q$ along $\s$.
Because $P_1$ is not cone-like, we can choose a facet $s_1$ of $P_1$
which is disjoint from $\s$. Then there is an $n$-cell $P_2$
adjacent to $P_1$ along $s_1$. Because $P_2$ is not cone-like, $P_2$
has a facet $s_2$ which is disjoint from $s_1$. Continuing in this
manner we obtain two infinite sequences $\{P_j\}$ of $n$-cells and
$\{s_j\}$ of facets such that $P_j\cap P_{j+1}=s_j$ for all $j\ge0$,
where we set $P_0=Q$ and $s_0=\s$. See Figure~\ref{fig:gallery}.
This motivates the following definition:
\begin{figure}[htbp]
\includegraphics[width=1\linewidth]{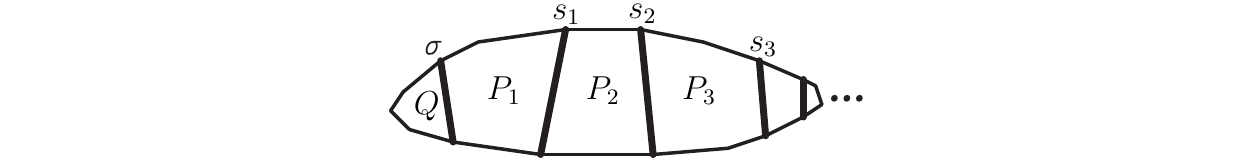}
\caption{\footnotesize A directed gallery from $Q$ in the direction
of $\s$.} \label{fig:gallery}
\end{figure}
\begin{definition}[Directed galleries] \label{def:gal}
A \emph{directed gallery} $Gal_{(Q,\s)}(P_j,s_j)$ from $Q$ in the
direction of $\s$ is the union $\bigcup_{j=0}^\infty P_j$ of an
infinite family of $n$-cells of $X$ such that for each $j\ge0$
\begin{itemize}
\item
$P_j\cap P_{j+1}=s_j$ is a facet of $X$, where $P_0=Q$ and $s_0=\s$;
\item
$s_j\cap s_{j+1}=\emptyset$.
\end{itemize}
\end{definition}

Thus the previous discussion says that to each facet $\s$ of $Q$ we
can associate a directed gallery $Gal_{(Q,\s)}(P_j,s_j)$ from $Q$ in
the direction of $\s$. Of course, because of the choices of $s_j$ we
made, the directed galleries are not uniquely determined by $Q$ and
$\s$. The lemma below, however, shows that they satisfy a common
property in relation to the iterated stars $st^j(Q)$ of $Q$ in $X$.

First notice the following. Because $X$ is strongly residually
convex by Theorem~\ref{thm:notriangle}, the proof of
Theorem~\ref{thm:main} applied to $X$ and $Q$ shows that the
iterated stars $st^j(Q)$ are convex $n$-polyballs. Recall that $X$
is assumed to be a subset of $\S^n$. Thus the stars $st^j(Q)$ form a
nested sequence of closed $n$-dimensional convex proper subsets of
$\S^n$.

\begin{lemma} \label{lem:gal}
Let $Gal_{(Q,\s)}(P_j,s_j)$ be a directed gallery from $Q$ in the
direction of $\s$. Then the following assertions are true:
\renewcommand{\labelenumi}{\textup{(\arabic{enumi})}}
\begin{enumerate}
\item
Each facet $s_j$ ($j\ge0)$ is in the boundary of the star $st^j(Q)$
of $Q$.
\item
The gallery $Gal_{(Q,\s)}(P_j,s_j)$ is a convex subset of $\S^n$.
\end{enumerate}
\end{lemma}

\begin{proof}
(1) The proof is by induction on $j\ge0$. When $j=0$, it is clear
that $s_0=\s$ is in the boundary of $st^0(Q)=Q$. Now assume that the
conclusion is true up to the $(j-1)$-th step. We need to show that
$s_j$ is in the boundary of $st^j(Q)$.

Because $P_j$ intersects $st^{j-1}(Q)$ at $s_{j-1}$, we have that
$$
s_j\subset P_j\subset st[st^{j-1}(Q)]=st^j(Q).
$$
To show that $s_j$ is in the boundary of $st^j(Q)$, consider the
residue $res(s_j)=P_j\cup P_{j+1}$. It is a convex subset of $\S^n$
by residual convexity of $X$. Moreover, it contains $s_{j-1}$ in its
boundary, since $s_{j-1}$ is a facet of $P_j$:
$$
s_{j-1}\subset\partial(P_j\cup P_{j+1})=\partial res(s_j).
$$
However, $s_{j-1}$ is disjoint from $s_j$ and hence from $P_{j+1}$.
Because $P_j$ is a convex polytope, it follows that $s_{j-1}$ is a
maximal convex subset in the boundary of $res(s_j)$. Now, by the
induction hypothesis, $s_{j-1}$ is also contained in the boundary of
the convex subset $st^{j-1}(Q)$. From the convexity of $res(s_j)$
and $st^{j-1}(Q)$, and from the maximality of $s_{j-1}$, it follows
that
$$
res(s_j)\cap st^{j-1}(Q)=s_{j-1}.
$$
Thus $P_{j+1}$ is disjoint from $st^{j-1}(Q)$ and cannot intersect
the interior of the star $st^j(Q)$. In particular, $s_j\subset
P_{j+1}$ does not intersect the interior of $st^j(Q)$ and hence must
be in the boundary of $st^j(Q)$. The induction is complete.

(2) Let $G_k=\bigcup_{j=0}^k P_j$. The previous proof of (1) shows
that $G_k$ is contained in $st^k(Q)$ and intersects $P_{k+1}$
exactly along $\s_k$. These facts inductively imply that $G_k$ is an
$n$-polyball for all $k\ge0$. Now, fix $k$ and let $e$ be a ridge in
the boundary of $G_k$. From the construction of galleries, it is
clear that $e$ intersects either a single $n$-cell of $G_k$ or two
adjacent $n$-cells of $G_k$. In either case, the residual convexity
of $X$ implies that the link $\L(e;G_k)$ is convex. Since $e$ is
arbitrary, it follows from Lemma~\ref{lem:ridge} that the polyball
$G_k$ is convex. Since $k$ is arbitrary, $G_k$ is convex for all
$k\ge0$. Because the nested sequence $\{G_k\}$ exhausts the gallery
$Gal_{(Q,\s)}(P_j,s_j)$, the conclusion follows.
\end{proof}

Recall that we fixed a specified $n$-cell $Q$ in $X$. Let $\s$ be a
facet of $Q$. Let $Gal_{(Q,\s)}(P_j,s_j)$ be a directed gallery from
$Q$ in the direction of $\s$. The above lemma says that each facet
$s_j$ in this gallery is in the boundary of the star $st^j(Q)$.
Denote by $\langle s_j\rangle$ the hyperplane spanned by $s_j$.
Because $st^j(Q)$ is convex, $\langle s_j\rangle$ must be a
supporting hyperplane of $st^j(Q)$. Now consider the sequence
$\{\langle s_j\rangle\}$ of hyperplanes of $\S^n$. Because $\S^n$ is
compact this sequence converges to a hyperplane which we denote by
\begin{align} \label{eqn:support}
H_X(\s).
\end{align}
Because the convex sets $st^j(Q)$ exhaust $X$ and their supporting
hyperplanes $\langle s_j\rangle$ converge to $H_X(\s)$, it
immediately follows that $H_X(\s)$ is a supporting hyperplane of the
convex subset $X\subset\S^n$. See Figure~\ref{fig:support}~(a).
\begin{figure}[htbp]
\includegraphics[width=1\linewidth]{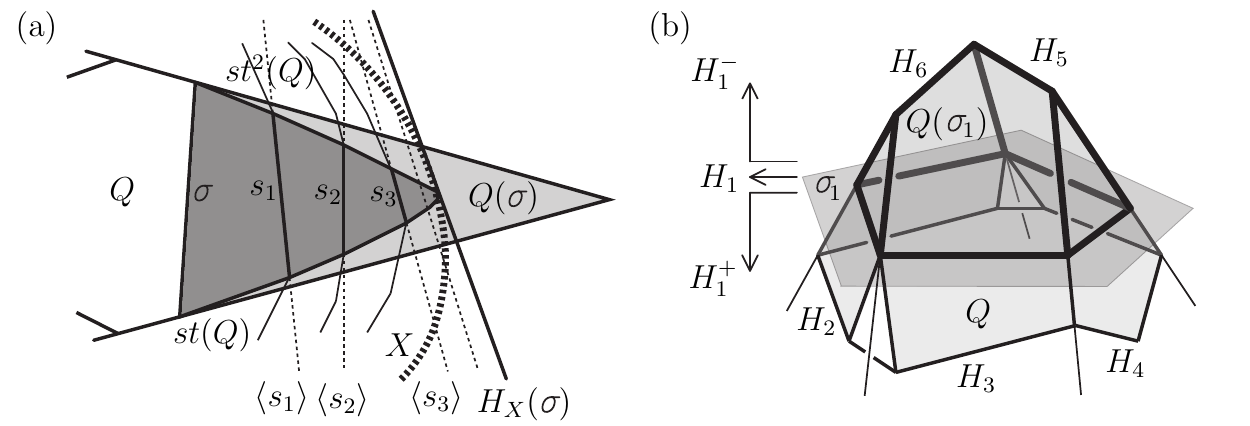}
\caption{\footnotesize (a) To each facet $\s$ of $Q$ we can
associate a directed gallery $Gal_{(Q,\s)}(P_j,s_j)$ in the
direction of $\s$, which again determines a supporting hyperplane
$H_X(\s)$ of $X$. (b) To each facet $\s_1$ of $Q$ we can assign a
cone-like polytope $Q(\s_1)$.} \label{fig:support}
\end{figure}

In this manner, to each facet $\s$ of $Q$, we can assign a
supporting hyperplane $H_X(\s)$ of $X$. Notice that the hyperplane
$H_X(\s)$ is not uniquely determined by the facet $\s$ because the
associated gallery is not uniquely determined by $\s$ either.
Therefore, we are rather interested in all possible locations of
$H_X(\s)$ in $\S^n$. As will be explained below, the restriction on
their location is given by the specified $n$-cell $Q$ and its
facets.

We may assume that the $n$-cell $Q\subset X$ is expressed as
$$
Q=\bigcap_{i=1}^m H_i^+,
$$
where $m\ge n+1$ and the $\{H_i^+\}$ is an irredundant family of
halfspaces bounded by hyperplanes $H_i$. Then the facets $\s_i$ of
$Q$ are of the form $\s_i=Q\cap H_i$ for $1\le i\le m$.

Consider the facet $\s_1$ of $Q$. Let $\s_2,\s_3,\ldots,\s_k$
($k<m$) be the facets of $Q$ that are adjacent to $\s_1$ along
ridges. Consider the $n$-polytope $Q(\s_1)$ defined as the
intersection of the $k$ halfspaces $H_1^-,H_2^+,\ldots,H_k^+$:
\begin{align} \label{eqn:q}
Q(\s_1)=H_1^-\cap H_2^+\cap\cdots\cap H_k^+.
\end{align}
See Figure~\ref{fig:support} (b). (The polytope $Q(\s_1)$ is
cone-like and its vertices in $\s_1$ are simple.) Consider also a
directed gallery $Gal_{(Q,\s_1)}(P_j,s_j)$ from $Q$ in the direction
of $\s_1$. By Lemma~\ref{lem:gal}~(2), it is a convex subset of
$\S^n$. However, the $(k-1)$ hyperplanes $H_2,\ldots,H_k$ support
$Q$ and hence $Gal_{(Q,\s_1)}(P_j,s_j)$. It follows that the set
$Gal_{(Q,\s_1)}(P_j,s_j)\setminus Q$ is contained in the polytope
$Q(\s_1)$:
$$
\bigcup_{j=1}^\infty P_j=Gal_{(Q,\s_1)}(P_j,s_j)\setminus Q\subset
Q(\s_1).
$$
Recall that each hyperplane $\langle s_j\rangle$ supports the star
$st^j(Q)$ of $Q$. Thus no $\langle s_j\rangle$ ($j\ge1$) can
intersect a neighborhood of $Q$ (namely, the interior of $st(Q)$)
but always intersects the interior of $Q(\s_1)$. Being the the limit
of the hyperplanes $\langle s_j\rangle$, the hyperplane $H_X(\s_1)$
cannot intersect $Q$ but must intersect $Q(\s_1)$. See
Figure~\ref{fig:support} (a).

If we define $Q(\s_i)$ analogously for each facet $\s_i$ of $Q$, the
analogous statements hold for the hyperplanes $H_X(\s_i)$:

\begin{lemma} \label{lem:restriction1}
Given an $n$-cell $Q$ in $X$, the hyperplanes $H_X(\s_i)$ and the
$n$-polytopes $Q(\s_i)$ associated to facets $\s_i$ of $Q$ satisfy
the following relations: for all $i$,
$$
H_X(\s_i)\cap Q=\emptyset\quad\text{and}\quad H_X(\s_i)\cap
Q(\s_i)\neq\emptyset.
$$
\end{lemma}

These restrictions on the location of $H_X(\s_i)$ are more
conveniently described in terms of duality, since the duals of the
halfspaces determined by $H_X(\s_i)$ are just points. The next
subsection is devoted to this description.

\subsection{Pavilions and $n+1$ hyperplanes in general position} \label{ssec:pav}

We continue to assume that $X\subset\S^n$ is a residually convex
$n$-complex such that none of its $n$-cells is triangular and that
$Q$ is a fixed $n$-cell in $X$. In our previous discussion we
expressed the $n$-cell $Q$ as
$$
Q=\bigcap_{i=1}^m H_i^+.
$$
Now, denote by $v_i=(H_i^+)^*$ the dual of the halfspace $H_i^+$.
Then each $v_i$ becomes a vertex of the dual polytope $Q^*$ of $Q$
(see Section~\ref{ssec:dual}):
$$
Q^*=\left[\bigcap_{i=1}^m
H_i^+\right]^*=\textup{conv}\left[\bigcup_{i=1}^m
(H_i^+)^*\right]=\textup{conv}\{v_1,v_1,\ldots,v_m\}.
$$
Recall also the definition \eqref{eqn:q} of the $n$-polytope
$Q(\s_1)$ associated to the facet $\s_1$ of $Q$:
$$
Q(\s_1)=H_1^-\cap H_2^+\cap\cdots\cap H_k^+.
$$
Its dual $Q(\s_1)^*$ is the convex hull of the vertices
$v_2,\ldots,v_k$ and $-v_1$, where $-v_1=(H_1^-)^*$ is the antipodal
point of $v_1$ (see Section~\ref{ssec:dual}):
\begin{align*}
Q(\s_1)^*&=\left(H_1^-\cap H_2^+\cap\cdots\cap
H_k^+\right)^*\\&=\textup{conv}\left\{(H_1^-)^*,(H_2^+)^*,\ldots,(H_k^+)^*\right\}\\
&=\textup{conv}\{-v_1,v_2,\ldots,v_k\}.
\end{align*}
Note that the vertices $v_2,\ldots,v_k$ of $Q(\s_1)^*$ (and $Q^*$)
are connected to $v_1$ by the edges of $Q^*$ which are dual to the
ridges $\s_1\cap\s_i$ ($2\le i\le k$) of $Q$.

Recall the definition \eqref{eqn:support} of the supporting
hyperplane $H_X(\s_1)$ of $X$ associated to the facet $\s_1$ of $Q$.
Now let $H_X(\s_1)^+$ be the halfspace which is bounded by
$H_X(\s_1)$ and which contains the $n$-complex $X$. Denote by
$x(\s_1)=[H_X(\s_1)^+]^*$ the dual point of $H_X(\s_1)^+$. In
Lemma~\ref{lem:restriction1} we summarized the restrictions on the
position of $H_X(\s_1)$. Dualizing these we obtain the following
conditions on the location of $x(\s_1)$:
\begin{quote}
Because $H_X(\s_1)^+$ contains $Q$ but $H_X(\s_1)$ does not
intersect $Q$, the point $x(\s_1)$ must be in the interior of $Q^*$.
On the other hand, because $H_X(\s_1)$ intersects $Q(\s_1)$, the
point $x(\s_1)$ cannot be an interior point of $Q(\s_1)^*$.
\end{quote}
These restrictions on $x(\s_1)$ motivate the following definition.
Recall that $S^\circ$ denotes the interior of a set $S$.

\begin{definition}[Pavilion] \label{def:pav}
Let $P$ be an $n$-polytope in $\S^n$. Let $v$ be a vertex of $P$ and
let $V(v)$ be the set of all vertices of $P$ that are connected to
$v$ by edges of $P$. Denote by $P(v)=\textup{conv}\left(\{-v\}\cup
V(v)\right)$ the convex hull of $-v$ and $V(v)$. The \emph{pavilion}
$\textup{pv}(v;P)$ of $v$ in $P$ is by definition
$$
\textup{pv}(v;P)=P^\circ\setminus P(v)^\circ.
$$
The \emph{base} $\underline{\textup{pv}}(v;P)$ of the pavilion
$\textup{pv}(v;P)$ is defined as
$$
\underline{\textup{pv}}(v;P)=P^\circ\cap\partial P(v).
$$
Note that the base $\underline{\textup{pv}}(v;P)$ is an open subset
of $\partial P(v)$. See Figure~\ref{fig:pav}.
\end{definition}

\begin{figure}[htbp]
\includegraphics[width=1\linewidth]{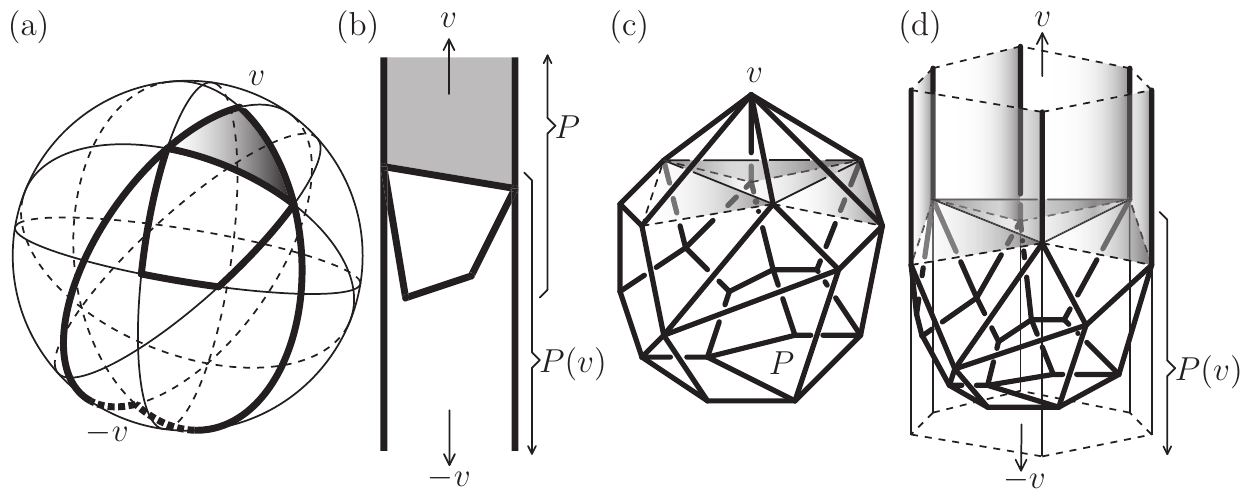}
\caption{\footnotesize Pavilions. (a) $P$ is a pentagon with
$\textup{pv}(v;P)$ shaded. (b) A view of (a) in an affine $2$-plane.
(c) $P$ is a $3$-polytope. The base $\underline{\textup{pv}}(v;P)$
of a pavilion is shaded and consists of four triangles. (d) A view
of (c) in an affine $3$-plane. The six facets containing $v$
determine a hexagonal cylinder. The pavilion $\textup{pv}(v;P)$ is
shaded.} \label{fig:pav}
\end{figure}

To summarize, the point $x(\s_1)=[H_X(\s_1)^+]^*$ we considered
above must be in the pavilion $\textup{pv}(v_1;Q^*)$ of $v_1$ in
$Q^*$. Similarly, by considering the analogous restrictions on
$H_X(\s_i)$ with respect to $Q$ and $Q(\s_i)$ given by
Lemma~\ref{lem:restriction1}, we obtain the following.

\begin{lemma} \label{lem:restriction2}
Let $Q$ be an $n$-cell in $X$. For each facet $\s_i$ of $Q$ ($1\le
i\le m$), let $H_X(\s_i)$ be the supporting hyperplanes of $X$
associated to $\s_i$. Let $H_X(\s_i)^+$ denote the halfspace which
is bounded by $H_X(\s_i)$ and which contains the $n$-complex $X$.
Then, for all $i$, the dual points of $H_X(\s_1)^+$
$$
x(\s_i)=[H_X(\s_i)^+]^*
$$
must satisfy
$$
x(\s_i)\in\textup{pv}(v_i;Q^*).
$$
\end{lemma}

Assuming another Lemma~\ref{lem:thin} below, we are now ready to
prove Theorem~\ref{thm:pconvexity}.

\begin{proof}[Proof of Theorem~\ref{thm:pconvexity}]
Let $Q$ be the $n$-cell of $X$ whose dual $Q^*$ is thick. As before,
we may assume that the $n$-cell $Q\subset X$ is expressed as
$$
Q=\bigcap_{i=1}^m H_i^+,
$$
where $m\ge n+1$ and the $\{H_i^+\}$ is an irredundant family of
halfspaces bounded by hyperplanes $H_i$. Then the facets $\s_i$ of
$Q$ are of the form $\s_i=Q\cap H_i$ for $1\le i\le m$ and the
vertices of the dual $Q^*$ are $v_i=(H_i^+)^*$.

As in Section~\ref{ssec:gallery}, for each facet $\s_i$ of $Q$, we
choose a directed gallery from $Q$ in the direction of $\s_i$ to
obtain a supporting hyperplane $H_X(\s_i)$ of $X$. We let
$x(\s_i)=[H_X(\s_i)^+]^*$ be the dual point of $H_X(\s_i)^+$. Then
Lemma~\ref{lem:restriction2} tells us that
$$
x(\s_i)\in\textup{pv}(v_i;Q^*)
$$
for all $1\le i\le m$.

Suppose by way of contradiction that the $m$ points $x(\s_i)$ are
contained in a hyperplane $H\subset\S_n$. Then $H$ necessarily
intersects all pavilions $\textup{pv}(v_i;Q^*)$ in $Q^*$. However,
Lemma~\ref{lem:thin} below implies that if this is the case then the
polytope $Q^*$ must be thin, contrary to our assumption. Therefore,
no hyperplane can contain all $m$ points $x(\s_i)$ simultaneously.

Hence there are some $n+1$ points $x(\s_i)$ in general position,
that is, they are not contained in a common hyperplane. This fact
again implies that there are $n+1$ supporting hyperplanes
$H_X(\s_i)$ of $X$ that are in general position, that is, their
intersection is empty. Then the $n+1$ supporting hyperplanes
$H_X(\s_i)$ determine an $n$-simplex in $\S^n$, which contains $X$.
Therefore, the $n$-complex $X$ must be a properly convex subset of
$\S^n$ and this completes the proof of Theorem~\ref{thm:pconvexity}.
\end{proof}

\begin{lemma} \label{lem:thin}
An $n$-polytope $P$ in $\S^n$ is thin provided that there exists a
hyperplane $H\subset\S^n$ which intersects all pavilions
$\textup{pv}(v;P)$ of vertices $v$ of $P$.
\end{lemma}

\begin{proof}
Let $H$ be a hyperplane which intersects all pavilions
$\textup{pv}(v;P)$ of vertices $v$ of $P$. There are two
possibilities depending on whether or not $H$ intersect the
interiors of all pavilions $\textup{pv}(v;P)$.

\emph{Case I}. Suppose that $H$ intersect the interiors of all
pavilions $\textup{pv}(v;P)$. Then we can perturb $H$ slightly so
that $H$ still intersects all pavilions $\textup{pv}(v;P)$ but
contains no vertices of $P$. Let $v$ be a vertex of $P$. Then $v$ is
in one halfspace, say $H^+$, determined by $H$. We need to show that
there is a vertex $v'\in V(v)$ which is in the other halfspace
$H^-$. Suppose on the contrary that all vertices of $V(v)$ are in
$H^+$. Because no vertex of $P$ is in $H$, we have that both $v$ and
$V(v)$ are in the interior of $H^+$. Thus the convex hull
$\textup{conv}\left(\{v\}\cup V(v)\right)$ is also in the interior
of $H^+$ and this gives a contradiction because we have
$$
\textup{pv}(v;P)\subset\textup{conv}\left(\{v\}\cup V(v)\right)
$$
and the pavilion $\textup{pv}(v;P)$ cannot intersect $H$. Therefore,
there is a vertex $v'\in V(v)$ which is in the halfspace $H^-$.
Since $v$ is arbitrary, this shows that $H$ is a cutting plane for
$P$ and hence $P$ is thin.

\emph{Case II}. Suppose that $H$ does not intersect the interior of
some pavilion $\textup{pv}(v_0;P)$. Note that the base
$\underline{\textup{pv}}(v_0;P)$ is an open subset of $\partial
P(v_0)$ and $\partial P(v_0)$ is \emph{concave} toward
$\textup{pv}(v_0;P)$. Thus, in this case, the base
$\underline{\textup{pv}}(v_0;P)$ has to be flat so that
$$
H\cap \textup{pv}(v_0;P)=\underline{\textup{pv}}(v_0;P)
$$
and hence the set $V(v_0)$ also has to be in $H$, that is,
$V(v_0)\subset H$. Let $v_0\in H^+$ without loss of generality.
Because $P$ is a (convex) polytope, this implies that those vertices
of $P$ which are not in $\{v_0\}\cup V(v_0)$, if any, have to be in
the interior of the halfspace $H^-$. There are two subcases to be
considered:

(1) If there is such a vertex $v_1$ of $P$, then we must have that
$V(v_1)=V(v_2)$ because otherwise the base
$\underline{\textup{pv}}(v_1;P)$ of the pavilion is contained in the
interior of $H^-$ and hence the pavilion $\textup{pv}(v_1;P)$ cannot
intersect $H$. It follows that $P$ is a bipyramid with tips
$\{v_0,v_1\}$ and with base the $(n-1)$-polytope
$\textup{conv}V(v_0)=\textup{conv}V(v_1)$. Now, we can perturb $H$ a
little bit so that $H$ still separates $v_0$ and $v_1$ and so that
$H$ does not intersect $V(v_0)$ but intersects the interior of
$\textup{conv}V(v_0)$. Then $H$ becomes a cutting plane for $P$.

(2) If there is no such vertex, then $P$ is a pyramid with apex
$v_0$ over the $(n-1)$-polytope $\textup{conv}V(v_0)$. In this case,
if we push $H$ slightly toward the apex $v_0$ then $H$ becomes a
cutting plane for $P$.

Therefore, in both subcases, $P$ has a cutting plane and is
necessarily thin.
\end{proof}

\subsection{Speculations}

In this subsection we shall again consider those residually convex
$n$-complexes which contain no triangular $n$-cells. We speculate
upon other approaches to proper convexity than the one provided by
Theorem~\ref{thm:pconvexity}.

Figure~\ref{fig:pentagon} (c) illustrates
Theorem~\ref{thm:pconvexity}: the given residually convex
$2$-complex $X\subset\S^2$ consists only of quadrilaterals and a
single pentagon $Q$. The dual of $Q$ is again a pentagon and, hence,
is thick. Thus $X$ satisfies the assumption of
Theorem~\ref{thm:pconvexity} and must be properly convex. Indeed,
since the rest of polygons in $X$ other than $Q$ are quadrilaterals,
each edge $\s_i$ ($1\le i\le5$) of $Q$ uniquely determines a gallery
in the direction of $\s_i$ from $Q$, which takes up the whole
triangle $Q(\s_i)$ and which uniquely determines a supporting line
$H_X(\s_i)$ (see Section~\ref{ssec:gallery}). Among those five
supporting lines, two pairs of them coincide but, as guaranteed by
the proof of the theorem, the remaining distinct three are in
general position bounding a $2$-simplex, whose interior is equal to
the $2$-complex $X$ in this case.

On the other hand, Figure~\ref{fig:pentagon} (b) explains why the
thickness condition is necessary: the given residually convex
$2$-complex $X\subset\S^2$ consists only of quadrilaterals. Each
quadrilateral in $X$ uniquely determines four supporting lines to
$X$, but two pairs of them always coincide to give rise to only two
distinct supporting lines to $X$. The $2$-complex $X$ is equal to
the domain bounded by the two supporting lines and hence is not
properly convex.

However, a generic residually convex $2$-complex without triangles
looks like the one in Figure~\ref{fig:pentagon} (a), which consists
of quadrilaterals and pentagons. In fact, one can obtain such a
generic $2$-complex using only quadrilaterals. See
Figure~\ref{fig:strict1} (b) and compare with the non-generic
example in Figure~\ref{fig:strict1} (a). This fact implies that
there are other causes than thickness which force complexes to be
properly convex. Observe that, in contrast with the complexes in
Figure~\ref{fig:pentagon} (b) and (c), the complex in
Figure~\ref{fig:pentagon} (a) has the following property. In
general, the underlying set of the star $st^k(Q)$ of a cell $Q$ can
be regarded as a polytope. The combinatorial complexity of the
polytope $st^k(Q)$ in Figure~\ref{fig:pentagon} (a) grows very fast
as $k$ goes to infinity. On the other hand, in
Figure~\ref{fig:pentagon} (b) and (c), the combinatorial complexity
of the stars $st^k(Q)$ is limited to only that of quadrilaterals or
pentagons. This observation raises the following issue:
\begin{figure}[htbp]
\includegraphics[width=1\linewidth]{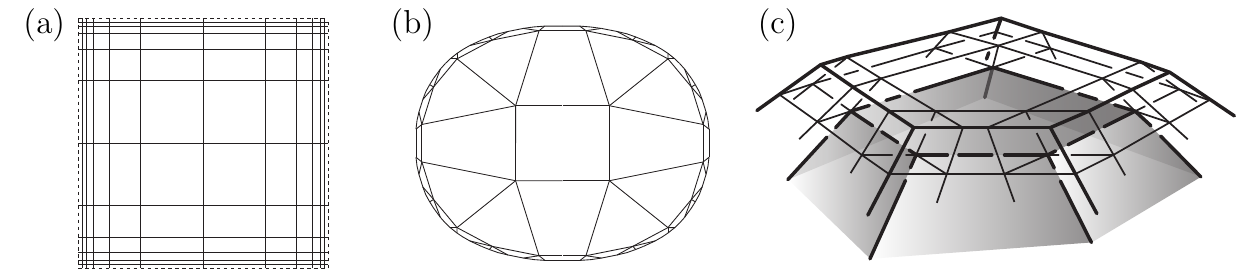}
\caption{\footnotesize (a) A properly convex domain consisting only
of squares. This example is uninteresting because it is a product of
two properly convex domains. (b) A properly convex domain. It is a
strictly residually convex $2$-complex and consists only of
quadrilaterals. The union of any two adjacent quadrilaterals is a
hexagon. (c) Using cubes and prisms, one can construct the stars
$st^k(Q)$ of a simple polytope $Q$ so that the combinatorics of
$st^k(Q)$ and $Q$ are the same.} \label{fig:strict1}
\end{figure}

Instead of considering those galleries starting from a fixed cell
$Q$ in the direction of its facets, we could also consider galleries
starting from facets in the boundary of a star $st^k(Q)$ for
sufficiently large $k$. If the combinatorial complexity of the stars
$st^k(Q)$ (viewed as polytopes) grows unlimitedly as $k$ increases,
then so too does the chance that there are many distinct supporting
hyperplanes associated to galleries starting from the facets in the
boundary of $st^k(Q)$, so that we can always choose $n+1$ such in
general position. Thus one may ask:
\begin{quote}
Question: \textsl{Find conditions which guarantee that the
combinatorial complexity of $st^k(Q)$ strictly increases as $k$
increases.}
\end{quote}
This question is interesting in view of the fact that properly
convex real projective structures behave very similarly to metric
spaces of non-positive curvature (see
Section~\ref{ssec:convex-rpn}); answers to this question can
possibly turn out to be restrictions on the fundamental domains and
the gluing maps for such spaces. A number of reasonable approaches
to this question are as follows:

(1) Figure~\ref{fig:strict1} (b) motivates the following condition
in addition to residual convexity: for each adjacent pair of
$n$-cells $P_1$ with $k_1$ facets and $P_2$ with $k_2$ facets, we
require that the underlying set of $P_1\cup P_2$ be an $n$-polytope
with $k_1+k_2-2$ facets. In other words, we require that no two
facets in the boundary of $P_1\cup P_2$ span a common hyperplane. We
may call this property as \emph{strict residual convexity}. In the
case when the notion of angle makes sense, this condition amounts to
not allowing right-angled polytopes.

Even when we do not require strict residual convexity, there are
other possible answers to the above question.

(2) As we observed in Figure~\ref{fig:pentagon} (b) and (c),
quadrilaterals are not good for our current purposes. Similar
examples are also possible in general dimension with $n$-cubes
taking the role of quadrilaterals, if $Q$ is a simple polytope. See
Figure~\ref{fig:strict1} (c). Even if we disallow $n$-cubes,
however, by taking product with a $2$-dimensional example, we may
obtain a complex consisting of $n$-prisms which is not properly
convex. It seems that a complex without cubes and with a non-prism
cell has good chance to be properly convex.

\begin{figure}[htbp]
\includegraphics[width=1\linewidth]{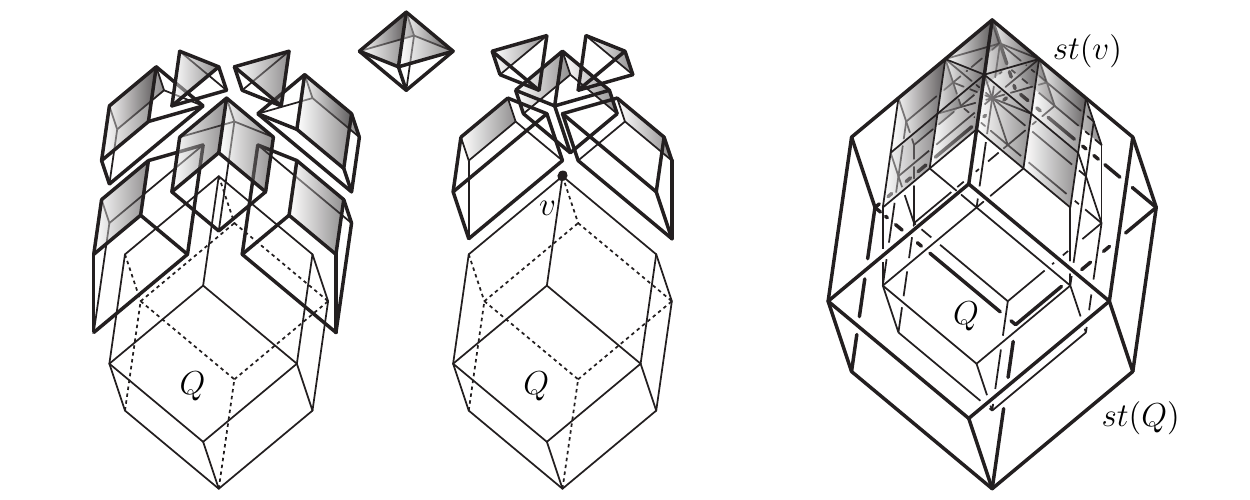}
\caption{\footnotesize The $3$-polytope $Q$ and its star $st(Q)$ are
rhombic dodecahedra. The star $st(v)$ of a non-simple vertex $v$ of
$Q$ contains four pyramids.} \label{fig:strict2}
\end{figure}

(3) Suppose that $X$ contains an $n$-polytope $Q$ which has a
non-simple vertex. Then it is very likely that the combinatorial
complexity of $st^k(Q)$ strictly increases as $k$ increases:
Figure~\ref{fig:strict2} exhibits a way to construct the star
$st(Q)$ of a rhombic dodecahedron $Q$ so that the combinatorics of
$Q$ and $st(Q)$ are the same. The rhombic dodecahedron $Q$ has
non-simple vertices. The star $st(v)$ of one of those non-simple
vertices is shown in the picture. Observe that the star $st(v)$
contains four tetrahedra, which are prohibited in our current
discussion because they are triangular. Moreover, the star $st(v)$
also contains some cubes. Thus it is very unlikely that one can
construct the star $st(Q)$ of a non-simple polytope $Q$ without
using triangular polytopes so that the combinatorics of $Q$ and
$st(Q)$ are the same, even though cubes are allowed.

\section{Applications to real projective structures} \label{sec:rpn}

In this section we introduce real projective structures and prove
Theorem~\ref{introthm:main}.

\subsection{Convex real projective structures}
\label{ssec:convex-rpn}

Let $X$ be a smooth manifold and $G$ a Lie group acting on $X$. An
\emph{$(X,G)$-structure} on a manifold $M$ is a maximal atlas
$\{(U_i, \phi_i)\}$ on $M$, where the family $\{U_i\}$ forms an open
covering of $M$ and the maps $\phi_i: U_i\to X$ are coordinate
charts such that the restriction of the transition map $ \phi_j
\circ \phi_i^{-1}$ to each component of $\phi_i(U_i\cap U_j)$ is the
restriction of an element of $G$. Let $M$ and $N$ be manifolds with
$(X,G)$-structures. A map $f:M\to N$ is an \emph{$(X,G)$-map} if,
for each pair of charts $\phi_i:U_i\to X$ and $\psi_j:V_j\to X$ for
$M$ and $N$, respectively, the restriction of the composition
$\psi_j\circ f\circ\phi_i^{-1}$ to each component of $\phi_i(U_i\cap
f^{-1}(V_j))$ is the restriction of an element of $G$.

Let $M$ be a manifold with $(X,G)$-structure. Let $p:\tilde{M}\to M$
be the universal covering space of $M$ and identify $\pi_1(M)$ with
the group of covering transformations. Then there is a unique
$(X,G)$-structure on $\tilde{M}$ for which $p$ is an $(X,G)$-map.
Furthermore, the Development Theorem (see \cite{gold-lec}) says that
there exists a pair $(dev,\rho)$ where $dev:\tilde{M}\to X$ is an
$(X,G)$-map and $\rho:\pi_1(M)\to G$ is a homomorphism such that
$$
dev\circ\gamma=\rho(\gamma)\circ dev
$$
for each $\gamma\in\pi_1(M)$. If $(dev',\rho')$ is another such
pair, there exists $g\in G$ such that $dev'=g\circ dev$ and
$\rho'(\gamma)=g\rho(\gamma)g^{-1}$ for each $\gamma\in\pi_1(M)$.

A \emph{real projective structure} is an $(X,G)$-structure where $X$
is the real projective space $\R\textup{P}^n$ and $G$ is the group
$\textup{Aut}(\R\textup{P}^n)$ of projective automorphisms. The
universal cover $\S^n$ of $\R\textup{P}^n$ is called the
\emph{projective $n$-sphere} and its group $\textup{Aut}(\S^n)$ of
projective automorphisms is isomorphic to the group
$SL^{\pm}(n+1,\R)$ of real matrices of determinant $\pm1$. A real
projective structure can also be defined as a
$(\S^n,\textup{Aut}(\S^n))$-structure (see \cite[Exercise
4.5]{gold-lec}). For the sake of convenience, we shall adopt the
latter as our definition of real projective structures.

Let $M$ be a \emph{real projective $n$-manifold}, that is, a
manifold with a real projective structure. If the developing map
$$
dev:\tilde{M}\to\S^n
$$
is an embedding onto a convex (resp. properly convex) domain
$\Omega\subset\S^n$ (see Section~\ref{sec:prelim}), then the
structure on $M$ is said to be \emph{convex} (resp. \emph{properly
convex}) and the manifold $M$ is called a \emph{convex} (resp.
\emph{properly convex}) real projective manifold.

Convex real projective structures enjoy some nice properties, which
we explain as follows. Let $M$ be a convex real projective
$n$-manifold. Then $M$ is isomorphic to the quotient
$\Omega/\Gamma$, where $\Omega\subset\S^n$ is a convex domain and
$\Gamma\subset\textup{Aut}(\S^n)$ is a discrete subgroup acting
freely and properly discontinuously on $\Omega$. In particular, the
fundamental group of $M$ is identified with $\Gamma$ and hence
linear. Furthermore, because $\Omega$ is convex, the universal cover
of $M$ is contractible and any two points $x$ and $y$ of $M$ can be
connected by a line segment which is the projection of a line
segment in $\Omega$ connecting a lift $\tilde{x}\in\Omega$ of $x$ to
a lift $\tilde{y}\in\Omega$ of $y$. This property resembles the
notion of geodesic completeness of Riemannian metrics. For this
reason, convex real projective structures can be regarded as natural
analogues of complete Riemannian metrics.

Properly convex real projective structures are expected to resemble
non-positively curved metrics. For example, Benoist \cite{benoist,
benoist4} showed the followings: Let $M$ be a compact properly
convex real projective $n$-manifold. As above, represent $M$ as the
quotient $M=\Omega/\Gamma$, where $\Omega\subset\S^n$ is a properly
convex domain and $\Gamma\subset\textup{Aut}(\S^n)$ acts on $\Omega$
cocompactly. Then $\Omega$ is strictly convex if and only if
$\Gamma$ is Gromov-hyperbolic. (Here, strict convexity of $\Omega$
means that the boundary $\partial\Omega$ does not contain any open
line segment.) Furthermore, if $n=3$ and $\Omega$ is neither
strictly convex nor reducible, then $M$ admits the JSJ-decomposition
along embedded tori into hyperbolic pieces. In particular, such $M$
admits a Riemannian metric of non-positive curvature (see
\cite{leeb}).

\subsection{Obtaining real projective manifolds} \label{sec:manifold}

In this section, we present a version of the Poincar\'e fundamental
polyhedron theorem for real projective structures, which will
complement our main theorem in Section~\ref{ssec:convexity-theorem}.

Let $\mathcal{P}$ be a finite family of $n$-polytopes in $\S^n$.
Denote by $\Sigma$ the collection of all facets of the polytopes in
$\mathcal{P}$. A \emph{projective facet-pairing} for $\mathcal{P}$
is a set
$$
\Phi=\{\phi_\s\in\textup{Aut}(\S^n)\,|\,\s\in\Sigma\}
$$
of elements of $\textup{Aut}(\S^n)$ indexed by $\Sigma$ such that
\begin{itemize}
\item
for each facet $\s$ of $P\in\mathcal{P}$ there is a facet $\s'$ of
$P'\in\mathcal{P}$ such that $\phi_\s(\s)=\s'$;
\item
the polytopes $\phi_\s(P)$ and $P'$ are situated so that
$\phi_\s(P)\cap P'=\s'$;
\item
the maps $\phi_\s$ and $\phi_{\s'}$ satisfy the relation
$\phi_{\s'}=\phi_\s^{-1}$.
\end{itemize}

Let $\Phi$ be a projective facet-pairing for $\mathcal{P}$. Then
$\Phi$ induces an equivalence relation on the disjoint union
$\Pi=\bigsqcup_{P\in\mathcal{P}}P$. The corresponding quotient space
$M$ of $\Pi$ is said to be obtained by gluing together the polytopes
of $\mathcal{P}$ by $\Phi$. Let $M'$ denote the space $M$ removed
with its cells of codimension $\ge2$. The space $M'$ has a natural
structure of a real projective orbifold, which is a manifold
provided that $\phi_{\s}(\s)\ne\s$ for every facet $\s\in\Sigma$.
While the following discussion has a straightforward generalization
in the context of real projective orbifolds, we assume, for
simplicity, that $M'$ is a real projective manifold.

In what follows, we shall obtain a necessary condition for the real
projective structure on $M'$ to extends to $M$ and for the space $M$
to be a real projective manifold.

For this purpose, note first that the equivalence relation on $\Pi$
also induces an equivalence relation on the collection of ridges of
the polytopes in $\mathcal{P}$. More precisely, let $e:=e_1$ be a
ridge of $P_1\in\mathcal{P}$. Choose a facet $\s_1$ of $P_1$
containing $e_1$. Then there is a facet $\s_1'$ of
$P_2\in\mathcal{P}$ such that $\phi_{\s_1}(\s_1)=\s_1'$. Let
$e_2=\phi_{\s_1}(e_1)$ and let $\s_2$ be the facet of $P_2$ other
than $\s_1'$ which contains $e_2$. Then there is a facet $\s_2'$ of
$P_3\in\mathcal{P}$ such that $\phi_{\s_2}(\s_2)=\s_2'$. Continuing
in this manner, we obtain a sequence $\{e_i\}$ of ridges, a sequence
$\{P_i\}$ of polytopes, a sequence $\{\phi_{\s_i}\}$ of
facet-pairing transformations, and a sequence $\{\s_i, \s_i'\}$ of
pairs of facets. Because the family $\mathcal{P}$ is finite and
there are only finitely many ridges in a polytope, the sequence of
ridges is periodic and hence all four sequences are periodic. Let
$r$ be the least common period of these four sequences. Note that
the period $r$, as well as the two conditions we shall consider
below, are independent of our choice above between $\s_1$ and
$\s_1'$.

We set $h(e)=\phi_{\s_r}\circ\cdots\circ\phi_{\s_1}$ and consider
the following sequence of polytopes in $\S^n$
$$
P_1,\;
\phi_{\s_1}^{-1}(P_2),\;\phi_{\s_1}^{-1}\phi_{\s_2}^{-1}(P_3),\;\cdots,\phi_{\s_1}^{-1}\phi_{\s_2}^{-1}\cdots\phi_{\s_{r-1}}^{-1}(P_r).
$$
Observe that all polytopes in the sequence share the ridge $e$ in
common and each successive polytopes are adjacent. Thus, if we put
the standard Riemannian metric on $\S^n$ and consider the link
$\L(e;P)$ for each polytope $P$ in the above sequence, then we
obtain a sequence $\{\alpha_i\}$ of segments in
$\S^1=L(e)^\bot\subset\S^n$. Let
$\L(e)=(\alpha_1\sqcup\cdots\sqcup\alpha_r)/_\sim$ denote the
natural identification space of these segments.

Now, for the space $M$ to be a real projective manifold, it is
necessary that, for each ridge $e$, we have $h(e)=id$ and the
isometry $L(e)=\S^1$. It turns out that these conditions are also
sufficient. The proof of the following proposition is analogous to
the usual proofs of the Poincar\'e fundamental polyhedron theorem
for constant curvature Riemannian metrics (see, for example,
\cite{ep} and \cite{rat}) and we omit it.

\begin{proposition} \label{prop:poincare}
Let $\mathcal{P}$ be a finite family of $n$-polytopes in $\S^n$. Let
$\Phi$ be a projective facet-pairing for $\mathcal{P}$. Let $M$ be
the space obtained by gluing together the polytopes of $\mathcal{P}$
by $\Phi$. Then $M$ is a real projective manifold provided that, for
each ridge $e$ of a polytope in $\mathcal{P}$, we have
\renewcommand{\labelenumi}{\textup{(\arabic{enumi})}}
\begin{enumerate}
\item
$h(e)=id$;
\item
$\L(e)$ is isometric to the unit circle $\S^1$.
\end{enumerate}
\end{proposition}

\subsection{Convexity theorem for real projective structures}
\label{ssec:convexity-theorem}

We are now ready to prove Theorem~\ref{introthm:main}. To apply the
results obtained in Section~\ref{sec:convexity} and
Section~\ref{sec:pconvexity} more conveniently, however, we prove
the following equivalent theorem which is stated in terms of
$(\S^n,\textup{Aut}(\S^n))$-structures.

\begin{theorem} \label{thm:mmain}
Let $\mathcal{P}$ be a finite family of $n$-polytopes in the
projective $n$-sphere $\S^n$. Let
$\Phi=\{\phi_\s\in\textup{Aut}(\S^n)\,|\,\s\in\Sigma\}$ be a
projective facet-pairing for $\mathcal{P}$, where $\Sigma$ is the
collection of all facets of the polytopes in $\mathcal{P}$. Let $M$
be a real projective $n$-manifold obtained by gluing together the
polytopes in $\mathcal{P}$ by $\Phi$. Assume the following
condition:
\begin{quote}
for each facet $\s$ of $P\in\mathcal{P}$, if $\s'$ is a facet of
$P'\in\mathcal{P}$ such that $\phi_\s(\s)=\s'$, then the union
$\phi_\s(P)\cup P'$ is a convex subset of $\S^n$.
\end{quote}
Then the following assertions are true:
\renewcommand{\labelenumi}{\textup{(\Roman{enumi})}}
\begin{enumerate}
\item
If $\mathcal{P}$ contains no triangular polytope, then $M$ is a
convex real projective manifold;
\item
If, in addition, $\mathcal{P}$ contains a polytope $P$ whose dual
$P^*$ is thick, then $M$ is a properly convex real projective
manifold.
\end{enumerate}
\end{theorem}

\begin{proof}
Let $dev:\tilde{M}\to\S^n$ be the associated developing map of the
universal covering space $\tilde{M}$ of $M$. Regard the projective
sphere $\S^n$ as the standard Riemannian sphere and pullback the
Riemannian metric to $\tilde{M}$ via $dev$. Then the above condition
on the facet-pairing for $\mathcal{P}$ and the assumption that $M$
is a real projective $n$-manifold, imply that $\tilde{M}$ is a
residually convex $n$-complex (as defined in
Definition~\ref{def:complex} and Definition~\ref{def:rconvexity}).
Now the conclusions of the theorem follow immediately from
Corollary~\ref{cor:notriangle} and Theorem~\ref{thm:pconvexity}.
\end{proof}

\begin{remark} \label{rem:vin-kap}
(1) It is not difficult to see that an orbifold version of
Theorem~\ref{thm:mmain} is also true.

(2) Let $P\subset\S^n$ be an $n$-polyhedron (which is not
necessarily a polytope). Suppose that
$\Gamma\subset\textup{Aut}(\S^n)$ is a group generated by
(projective) reflections in the hyperplanes spanned by facets of
$P$. In \cite{vinberg} Vinberg provided necessary and sufficient
conditions for $\Gamma$ to be a discrete subgroup with fundamental
domain $P$. In such case, he also showed that the orbit
$\Gamma(P)\subset\S^n$ of $P$ under $\Gamma$ is a convex subset and
$\Gamma$ acts properly discontinuously on the interior $\Omega$ of
$\Gamma(P)$.

It is easy to see that gluing by reflections necessarily gives rise
to residually convex structures. Thus, in some special cases, our
result provides another proof that the domain $\Omega$ above is
convex. Namely, if $P$ is a non-triangular $n$-polytope, if $\Gamma$
is known to be discrete, and if all stabilizer subgroups of points
of $P$ are finite, then $\Gamma(P)=\Omega$ is a residually convex
$n$-complex without triangular polytopes and hence
Corollary~\ref{cor:notriangle} applies.

On the other hand, because our gluing maps are not necessarily
reflections, our results do cover complementary part of Vinberg's
convexity assertion. For instance, it is well-known that
cocompact/cofinite hyperbolic reflection groups are non-existent in
higher dimensions. More generally, a similar non-existence assertion
is also true for cocompact (projective) reflection groups acting on
strictly convex domains (see \cite{js1}).

(3) In his paper \cite{kapovich}, after producing real projective
structures on Gromov-Thurston manifolds, Kapovich showed that these
structures are in fact convex. There he deals with polyhedral
complexes which are similar to our residually convex $n$-complexes.
But, because the polyhedra he considers have infinitely many facets,
his complexes are assumed to satisfy more properties than residual
convexity and are rather complicated to describe. His proof modifies
Vinberg's arguments and applies small cancelation theory to the
$2$-skeleton of the dual complexes.
\end{remark}

\medskip
Department of Mathematics,

University of California,

Davis, CA 95616, USA,

\medskip
\verb"zlee@math.ucdavis.edu"

\end{document}